\documentclass[12pt]{amsart}
\usepackage{amsfonts,amssymb,amsmath,amsthm}
\usepackage{xcolor}
\usepackage{fullpage}

\newcommand{\Rm}{{\mathbb R}}
\newcommand{\Pm}{{\mathbb P}}

\newcommand{\vol}{{\hbox{\scriptsize {\bf V}}}}
\newcommand{\rci}{{ w}}

\newcommand{\HA}{\mathcal H}
\newcommand{\co}{\colon}

\newcommand{\Events}{\Omega}
\newcommand{\event}{\omega}
\newcommand{\size}{n}
\def\R{\mathbb{R}}
\def\Z{\mathbb{Z}}

\newcommand{\ep}{\varepsilon}
\def\hat{\widehat}
 
\newcommand{\commentout}[1]{}
\newcommand{\E}{{\mathbb E}}

\def\N{\mathbb{N}}
\newcommand{\Range}{{\aleph}}
\newcommand{\RR}{\mathcal R}

\newcommand{\no}{\nonumber}
\newcommand{\abs}[1]{\lvert #1 \rvert}

\def\flux{\operatorname{flux}}
\def\pd{\partial}

\newcommand{\Var}{\operatorname{Var}}

\newcommand{\br}{\begin{eqnarray}}
\newcommand{\er}{\end{eqnarray}}
\newcommand{\be}{\begin{equation}}
\newcommand{\ee}{\end{equation}}
\newcommand{\baa}{\begin{array}}
\newcommand{\eaa}{\end{array}}
\newcommand{\ba}{\begin{eqnarray}}
\newcommand{\ea}{\end{eqnarray}}

\newtheorem{theorem}{\bf Theorem}[section]

\newtheorem{thm}[theorem]{Theorem}
\newtheorem{lem}[theorem]{Lemma}
\newtheorem{lemma}[theorem]{Lemma}
\newtheorem{prop}[theorem]{Proposition} 
\newtheorem{cor}[theorem]{Corollary}

\theoremstyle{definition}
\newtheorem{defin}[theorem]{Definition}
\newtheorem{rmk}[theorem]{Remark}

\theoremstyle{remark}
\newtheorem*{remark}{Remark}

\numberwithin{equation}{section}

\begin{document}

\title{Feeble fish in time-dependent waters \\ and homogenization of the G-equation} 

\author{Dmitri Burago}                                                          
\address{Dmitri Burago: The Pennsylvania State University,                          
Department of Mathematics, University Park, PA 16802, USA}                      
\email{burago@math.psu.edu}                                                     
                                                                                
\author{Sergei Ivanov}
\address{Sergei Ivanov:
St.\ Petersburg Department of Steklov Mathematical Institute,
Russian Academy of Sciences,
Fontanka 27, St.Petersburg 191023, Russia}
\email{svivanov@pdmi.ras.ru}

\author{Alexei Novikov}
\address{Alexei Novikov: The Pennsylvania State University,                          
Department of Mathematics, University Park, PA 16802, USA}                      
\email{anovikov@math.psu.edu}                                

\thanks{The first author was partially supported
by NSF grant DMS-1205597.
The second author was partially supported by
RFBR grant 17-01-00128.
The third author was partially supported
by NSF grants DMS-1515187 and DMS-1813943.
}



\begin{abstract}
We study the following control problem.
A fish with bounded aquatic locomotion speed swims in fast waters.
Can this fish, under reasonable assumptions, 
get to a desired destination? It can, even if the flow is time-dependent.
Moreover, given a prescribed sufficiently large time $t$, it can be 
there at exactly the time $t$. The major difference from our previous work is the time-dependence of the flow.
We also give an application to homogenization of the G-equation. 
\end{abstract}

\maketitle

\section{Introduction}

Let $V=V_t$ be a time-dependent vector field in $\R^\size$,  $\size \geq 2$.
We assume that $V_t(x)$ is continuous, uniformly bounded, and locally Lipschitz in~$x$.
We often abuse the language and refer to $V_t$ as a {\em flow}.

\begin{defin}\label{d:reach}
An absolutely continuous path $\gamma\co [t_0,t_1]\to\R^n$
is said to be {\em admissible} if
$$
  \left|\frac{d}{dt} \gamma(t) - V_t(\gamma(t))\right| \le 1
$$
for a.e. $t\in[t_0,t_1]$.

Let $x_0,x_1 \in \R^{\size}$, $t_0, t_1 \in \R$, $t_0 \leq t_1$.  
We say that a point $(x_1,t_1)$ in space-time is {\it reachable from} $(x_0,t_0)$
if there exists an admissible path $\gamma\co[t_0,t_1]\to\R^n$ with
$\gamma(t_0)=x_0$ and $\gamma(t_1)=x_1$.
%

If $(x_1,t_1)$ is reachable from $(x_0,t_0)$, we also say that  
$x_1$ is {\it reachable from $(x_0,t_0)$ at time $t_1$}.
In the sequel we usually assume that the initial conditions are $x_0=0$ and $t_0=0$.
For brevity, we say that $x$ is {\it reachable at time $t$} if
$(x,t)$ is reachable from $(0,0)$.
\end{defin}

We suggest the following naive interpretation of our set-up.
The vector field $V_t$ is the velocity field of waters in an ocean.
Fish living in the ocean have bounded aquatic locomotive speed.
We normalize the data so that the maximal speed of the fish is~1,
and the speed of waters can be much larger.
Definition \ref{d:reach} formalizes the condition that a fish
starting its journey from $x_0$ at time~$t_0$ can control its motion so that
it finds itself at $x_1$ at exactly time~$t_1$.

\medskip

A similar problem  was considered in~\cite{BIN, KS} for  time-independent vector fields $V$ and 
a  weaker reachability result: the fish is not required to arrive at its destination exactly at a prescribed time. 

Handling time-dependence of $V_t$ required considerable effort and actually forced us to prove 
a stronger result. This reachability problem is directly related to 
the G-equation which in particular
models combustion processes in the presence of turbulence.
Therefore another substantial part of this paper is an application to homogenization of the G-equation. 
We address this application in Section~\ref{G_e}.

Our main result, see Theorem \ref{newmain} below,
states that under natural assumptions on $V_t$
every point is reachable at all sufficiently large times.
The assumptions on $V_t$ are the following:
\begin{enumerate}
\item[(i)]
The field $V_t(x)$ is bounded: 
$$
M:=1+\sup_{t,x}|V_t(x)|<\infty,
$$
 and is locally Lipschitz in $x$.
\item[(ii)]
The flow is incompressible: $\operatorname{div} V_t = 0$ for all t.
\item[(iii)]
Small mean drift: 
\begin{equation}\label{dos}
 \lim_{L\to\infty} \sup_{t\in\R, x\in\R^{\size}}
 \left\| \frac{1}{L^{\size}} \int_{[0,L]^{\size}} V_t(x+y) \, dy \right\| = 0 .
\end{equation}
\end{enumerate}

All assumptions (i)-(iii) are essential.
First, the flow might have a sink towards which the flow runs faster
than the maximum possible speed the fish can swim. This issue is easily resolved by the assumption (ii)
that the flow is incompressible. Next, the velocity of the flow might point in one direction
and again it may have speed greater than the maximal speed of the fish. This obstruction is
resolved by the condition (iii) of small mean drift on the large scale. Finally, the flow could be so strong that the fish is carried to infinity in finite time.
The condition (i)   rules out this possibility. The condition (i) is also a technical assumption which is needed
to be able to formulate the problem formally.

It was a surprise to us that, under these modest
assumptions the fish can reach every destination point $x\in\R^n$.
Furthermore, there is some $t_x$ such that if $t\geq t_x$,
the fish can get to $x$ at exactly time~$t$.
We also prove an asymptotically optimal bounds for
the reach time, namely $t_x$ grows no faster than~$|x|$
as $|x|\to\infty$.

Now we are in a position to formulate our main result. 

\begin{theorem}
\label{newmain}
For every flow $V_t$ satisfying 
(i)--(iii) above and every $a>1$,
there exists $C>0$ such that
for all $x_0, x\in\R^{\size}$ and $t_0\in\R$, $(x, t)$ is reachable
from $(x_0,t_0)$ for every $t\ge t_0+a{|x-x_0|}+C$. 
\end{theorem}

\begin{remark}
The constant $C$ in Theorem \ref{newmain} depends on $a$ 
and parameters of the flow.
One can check that $C$ can be determined in terms of $a$,
the parameter $M$ from (i),
and the rate of convergence of the mean drift to zero in (iii).
\end{remark}

The small mean drift assumption~(iii) may be relaxed
at the expense of a weaker estimate on the reach time.
Namely we have the following.

\begin{cor}\label{new_cor}
Let $V_t$ be a flow satisfying (i), (ii), and
\be\label{dos_sod}
 \Delta:=\inf_{L>0} \sup_{t\in\R, x\in\R^{\size}}
 \left\| \frac{1}{L^{\size}} \int_{[0,L]^{\size}} V_t(x+y) \, dy \right\| <1 .
\ee
Then for every $a>\frac1{1-\Delta}$ there exists $C>0$ such that
for all $x_0, x\in\R^{\size}$ and $t_0\in\R$, $(x, t)$ is reachable
from $(x_0,t_0)$ for every $t\ge t_0+a{|x-x_0|}+C$. 
\end{cor}

The gist of the proof of Theorem~\ref{newmain} is: 
Fix a flow $V_t$ and assume without loss
of generality that $x_0=0$ and $t_0=0$.
For $t,r>0$ let $\RR_t$ denote the set of points reachable at time $t$
and $I_r$ the cube $[-r,r]^{\size}$ in~$\R^{\size}$.  
Our goal is to show that, for every fixed $r$ and for all sufficiently large $t$
the set $\RR_t$ contains~$I_r$.
We do this analyzing the volume of the intersection $\RR_t\cap I_r$
as a function of $t$.

The paper is organized as follows. In Section~\ref{prelim} we introduce our notation and main tools.
In particular, there we discuss isoperimetric inequalities, co-area formula, slicing, 
and certain regularity results such as rectifiability of the boundary of the reachable set.
Several important facts about BV-functions can be found in Appendix~\ref{BV}.
In Section~\ref{stages} we prove Theorem~\ref{newmain} and Corollary~\ref{new_cor}.
Sections~\ref{leak} and~\ref{middle} provide auxiliary estimates needed in the proof
of Theorem~\ref{newmain}.
In Section~\ref{G_e} we give an application of  Theorem~\ref{newmain} to the theory of random homogenization of the G-equation.

\subsection*{Some further directions}
In a discussion with the first author, Leonid Polterovich suggested 
to consider a similar problem where the fish is not a point but rather 
a region (think of an amoeba or a jelly-fish, for instance).
Leonid suggested the following symplectic formulation.
Let us say we are in $\R^2$ and the flow is Hamiltonian. 
This, of course, means that the area of the fish does not change
but its shape may change.
The fish has a fixed amount of Hofer's energy it can spend  to change the flow. 
 In two dimensions Hofer's energy is
 \[
 E(u) = \int_{-\infty}^\infty [\sup_{x \in \R^2}(\psi(x,t)) - \inf_{x \in \R^2}(\psi(x,t))] \, dt,
 \]
 where $\psi(x, t)$ is the stream-function (Hamiltonian) of the flow $u(x,t)$. 
Now the problem in question is as follows: Initially the fish 
sits in some ball, and it wants to get to another (destination ball) of the same size.
Leonid has made the following observation, which first sounds very counter-intuitive.
If the flow is constant (possibly very fast, no small mean drift),
the fish can get from any ball to a ball of the same size located in 
the direction  opposite to the flow and very far.
Using the same amount of Hofer's energy,
the fish can swim against an arbitrarily fast flow arbitrarily far away! 

We do not include a formal proof here.
 Here is an intuitive description.
Assume that the fish
has $M$ worth of Hofer's energy, where $M$ depends on the radius of the initial ball.
It spends $M/3$ of energy to stretch itself into a needle fish, or perhaps like an eel.
By that time, the flow has carried the fish far away just in the opposite direction of where it wants to arrive.
But now the fish can swim quite fast upstream (like eels do). Then it spends another $M/3$ of energy to
go back, through the ball where it wants eventually to end its journey, to a carefully chosen 
place well behind the
destination ball. After that, the flow carries the fish to where it dreams to arrive to, and the fish spend the
remaining $M/3$ of energy to re-assemble itself back into a round disc shape at 
exactly the time when the flow brings 
it to its destination.  

Many open problems are left.
First of all, even in dimension 2, this argument works for a constant flow only.
Of course, it suggests that much more is possible, but in general the flow can have diverging streams,
turbulence which may wrinkle the shape of the fish, etc. Even worse in dimension four. 
There may be phenomena related to non-squeezing and such. 
We did not invest enough time into thinking about this. 

Furthermore, a rather challenging goal is to find a more physical formulation for a fish which is a  ``more 
material''  region of changing shape (and its volume its almost conserved). The first naive idea
that comes to one's mind is to impose restrictions on the potential energy of the membrane (to keep
the amoeba in one piece, at least) and on kinetic energy (for it is still ``feeble''). 
We have not made any progress in this direction so far.

\section{Notation and preliminaries}\label{prelim}

Let $I_r=[-r,r]^{\size}$ denote the cube with edge length $2r$ centered at 0, 
$B_r(x)$ the Euclidean ball of radius $r$ centered at $x\in\R^{\size}$,
and $\vol_{\size}=|B_1(0)|$ the volume of the unit ball in $\R^{\size}$.
Occasionally we use $r=\infty$, with the convention that $I_\infty=B_\infty(x)=\R^{\size}$.


For $x_0\in\R^\size$, $t_0\in\R$ and $t\geq 0$, we denote by $\RR_t(x_0, t_0)$ 
the set of points reachable from $(x_0, t_0)$  at time $t_0+t$, see Definition \ref{d:reach}.
For brevity, let $\RR_t=\RR_t(0,0)$.

The volume of   $\RR_t \cap I_r$ is denoted by $\rci(r,t)$:
\begin{equation}\label{volume}
\rci(r,t) = |\RR_t \cap I_r| = \int_{I_r} \chi_{\RR_{t}}(x)\, dx ,
\end{equation}
where $\chi_{\RR_{t}}$ is the characteristic function of the reachable set $\RR_{t}$. 
The volume $\rci(r,t)$ is the main quantity of interest.

Recall that the maximum control in Definition \ref{d:reach} is bounded by~1.
Hence  $|x-x_0|\le Mt$ if $x$ is reachable from $(x_0,t_0)$ at time~$t_0+t$,
where $M$ is defined in the condition (i) above.
Therefore
\be\label{rGrowth_short}
  \RR_t \subset B_{tM}(0)\subset I_{tM}
\ee
for all $t>0$. Hence $\RR_t \cap I_r=\RR_t$ if $r\ge tM$.

We now define $s(r,t)\ge 0$, the perimeter of $\RR_t$ inside the cube $I_r$.  As we discuss below, $s(r,t)$ is essentially the $(\size-1)$-dimensional Hausdorff measure 
of the set $\partial\RR_t\cap I_r$. 
The formal definition is based on the notion of total variation for BV functions, see Appendix~\ref{BV},
in particular Definition~\ref{perim}.
Namely
$$
s(r,t) := P(\RR_t,I_r^\circ)=\Var(\chi_{\RR_t}, I_r^\circ),
$$
where $I_r^\circ$ is the interior of $I_r$.
Here the last expression is the variation of the characteristic function
$\chi_{\RR_t}$ in $I_r^\circ$, see Definition \ref{d:variation}.


Denote 
$$
 D_r(t):={\RR_t} \cap\pd I_r .
$$
The following lemma estimates the rate of change of the volume of $\RR_t$.
It is the main technical tool in our proof.

\begin{lem}\label{lemma_fact1}
For any fixed $r>0$,
\begin{equation}\label{vGrowth}
\frac{d}{dt}\rci(r,t) \geq s(r,t) - \flux(V_t,D_r(t))
\end{equation}
in the sense of distributions (with respect to $t$),
where $\flux(V_t,D_r(t))$ is the flux of the vector field $V_t$
through the $(\size-1)$-dimensional ``surface'' $D_r(t)\subset\pd I_r$.
Formally $\flux(V_t,D_r(t))$ is defined by
$$
  \flux(V_t,D_r(t)) = \int_{D_r(t)} V_t(x)\cdot \nu(x) \,dx
$$
where $\nu(x)$ is the outer normal to the boundary of the cube $I_r$ at a point $x\in\pd I_r$.

In the case $r=\infty$ we also have \eqref{vGrowth}, in the form
\be\label{vGrowth-infty}
 \frac{d}{dt}\rci(\infty,t) \geq s(\infty,t) .
\ee
\end{lem}

\begin{remark}
The inequalities \eqref{vGrowth} and \eqref{vGrowth-infty}
are easy to verify in the case when $V_t$ is smooth and
the boundary of $\RR_t$ is a smooth hypersurface transverse
to $\pd I_r$. In fact, in this case the inequalities turn into equalities.
Indeed, for a small $\delta>0$ the change from $\RR_t$
to $\RR_{t+\delta}$ is approximately the composition
of two operations:
First move the reachable set time $\delta$ along the flow
and then replace the resulting set by its $\delta$-neighborhood.
The first operation does not change the volume of the set since
the flow is incompressible. However, the volume of
the intersection with $I_r$ changes, it is reduced by the amount
of the flow that leaks out through the boundary of $I_r$.
This amount is approximately $\delta\cdot\flux(V_t,D_r(t))$.
On the second step, taking the $\delta$-neighborhood increases the volume by approximately
$\delta\cdot s(r,t)$,
since $s(r,t)$ is the area of the relevant part of the boundary of~$\RR_t$.
Passing to the limit as $\delta\to 0$ one obtains equalities in \eqref{vGrowth} and~\eqref{vGrowth-infty}.

This type of argument can be carried over to the general case if one shows
that $\RR_t$ has a rectifiable topological boundary
(compare with \cite[\S2]{BIN}). This approach would be quite
technical for a time-dependent flow. 
To avoid these technicalities, we use another formalization of 
the notion of surface area
and prove Lemma \ref{lemma_fact1} with appropriate machinery.
\end{remark}

\begin{proof}[Proof of Lemma \ref{lemma_fact1}]
The relation \eqref{vGrowth-infty} follows from \eqref{vGrowth}
and \eqref{rGrowth_short}.
To prove \eqref{vGrowth},
consider a family of functions
$u^\ep\co\R^n\times\R^+\to\R$, $\ep>0$, defined by
\be\label{controlrep2}
 u^\ep(x,t) = \sup \{ e^{-|y|/\ep} \mid \text{$y\in\R^n$ is such that $x\in\RR_t(y,0)$}\}.
\ee
Equivalently, one can set $u^\ep_0(x)=e^{-|x|/\ep}$ for all $x\in\R^n$ and define
\begin{equation}
\label{controlrep}
u^\ep(x,t) = \sup 
 \{ u^\ep_0(\gamma(0)) \mid \text{$\gamma\co[0,t]\to\R^n$ is an admissible path with $\gamma(t)=x$} \} ,
\end{equation}
see Definition \ref{d:reach}.
We need two properties of $u^\ep$:
For every fixed $\ep>0$, the function $u^\ep$ is locally Lipschitz
and it satisfies the following 
partial differential equation:
\be\label{g_eqn1}
\partial_t u^\ep + V_t \cdot \nabla u^\ep = |\nabla u^\ep|
\ee
for a.e.\ $x\in\R^n$ and $t>0$,
where $\nabla u^\ep$ denotes the gradient of $u^\ep$ with respect to the first argument.
The equation \eqref{g_eqn1} is called the {\em G-equation} associated to $V_t$.

The above properties are not hard to verify directly.
Alternatively, one can prove them using the theory of viscosity solutions, as follows.
The equation \eqref{g_eqn1} is a Hamilton-Jacobi equation with the Hamiltonian
\[
H(t,x,p) = -|p| + V_t \cdot p
\]
and the corresponding Lagrangian
\[
L(t,x,q)= \inf_{p \in \R^\size} \left[  p \cdot q - H(t,x,p) \right] = 
\begin{cases}
0, \hbox{ if } |q-V_t|\leq 1,\\
-\infty,  \hbox{ otherwise.}
\end{cases}
\]
By e.g.~\cite[Theorem 7.2]{Frank},
the function $u^\ep$ defined by \eqref{controlrep} is a viscosity solution of~\eqref{g_eqn1}
with the initial data $u^\ep(x,0)=u_0^\ep$.
For a definition, motivations, and derivation of viscosity solutions for optimal control problems see e.g.~\cite{Bardi}.
Since $u^\ep_0$ is bounded and uniformly continuous and $V_t$ is locally Lipschitz and bounded, 
the viscosity solution $u^\ep (x,t)$ is locally Lipschitz (by e.g. Lemma 9.2 in~\cite{Bressan}).
Furthermore, a viscosity solution satisfies the equation whenever it is differentiable
(see e.g.\ Proposition 1.9 on p.31 in~\cite{Bardi}).
Hence by Rademacher's Theorem $u^\ep$ satisfies~\eqref{g_eqn1} almost everywhere.


The formula \eqref{controlrep2} implies that
$u^{\ep}(x,t) \downarrow \chi_{\RR_{t}}(x)$ 
as $\ep\downarrow 0$,
where $\chi_{\RR_{t}}$ is the characteristic function of $\RR_t$. Hence
\[
\int_{I_r} u^{\ep}(x,t)\, dx  \to \rci(r,t)
\]
and
\[
 \flux(V_t u^\ep, \pd  I_r) \to  \flux(V_t,D_r(t))
\] 
as  $\ep \to 0$.
Integrating the G-equation over $I_r$ 
and taking into account the incompressibility of $V_t$
we obtain that
\[
\partial_t \int_{I_r} u^{\ep}\, dx + \flux(V_t u^\ep, \pd I_r) 
= \int_{I_r} |\nabla u^\ep|\, dx.
\]
Hence for any $t_1$ and $t_2$ we have
\[
\int_{t_1}^{t_2} \Var(u^\ep, I_r^\circ) \, dt
= \int_{t_1}^{t_2}  \int_{I_r} |\nabla u^\ep| \, dx dt 
= \int_{I_r} u^{\ep}(x, t_2) d x -   \int_{I_r} u^{\ep}(x, t_1) d x 
+ \int_{t_1}^{t_2} \flux(V_t u^\ep, \pd I_r) \,dt  .
\]
Note that this quantity is bounded by a constant independent of $\ep$
since $|u^\ep|\le 1$ and $|V_t|\le M$.
By Fatou's Lemma and the lower semi-continuity of the total variation (see e.g. Remark 3.5 in~\cite{A}) 
it follows that
\[
 \int_{t_1}^{t_2}  s(r,t) \,dt \equiv  \int_{t_1}^{t_2} \Var(\chi_{\RR_{t}}, I_r^\circ) \, dt 
 \leq 
 \liminf_{\ep\to 0} \int_{t_1}^{t_2}   \Var(u^\ep, I_r^\circ) \, dx dt .
\]
Thus
\[
\int_{t_1}^{t_2}  s(r,t) \,dt \le \ 
\rci(r, t_2) -\rci(r, t_1)  + \int_{t_1}^{t_2} \flux(V_t, D_r(t)) \,dt.
\]
This inequality means that \eqref{vGrowth} holds in the sense of distributions. 
\end{proof}

\begin{remark}
Since $\flux(V_t,D_r(t))$ is bounded for every fixed $r$ and $s(r,t)\ge 0$,
Lemma \ref{lemma_fact1} implies that $\rci(r,\cdot)$ is the sum of a Lipschitz function
and a non-decreasing function. Therefore  for almost all $t>0$ the derivative 
$\frac d{dt}\rci(r,t)$ exists and satisfies \eqref{vGrowth}.
\end{remark}

By \eqref{vGrowth-infty} the perimeter $P(\RR_t)=s(\infty,t)$ is finite for almost all $t>0$.
This and the De Giorgi Theorem~\ref{dg} imply that
the perimeter of ${\RR_t}$ equals the $(\size-1)$-dimensional 
Hausdorff measure $\HA^{n-1}(\pd^*\RR_t)$ of a rectifiable set $\partial^* {\RR_t}$, 
the reduced boundary of ${\RR_t}$ (see Definition~\ref{reduced}).
We define $p(r,t)$ to be the $(\size-2)$-dimensional Hausdorff measure of
the slice of $\partial^* {\RR_t}$ by $\pd I_r$:
\begin{equation}
p(r,t) = \HA^{\size-2} (\pd^* {\RR_t} \cap\pd I_r ).
\end{equation}
Then Corollary~\ref{s-fla} gives us 
the {\it co-area inequality} for this slicing:
\be\label{co-area}
  s(r_2,t)-s(r_1,t) \ge \int_{r_1}^{r_2} p(x,t)\,dx .
\ee
The quantity $p(r,t)$ can be though of as the $(\size-2)$-dimensional
perimeter of the $(n-1)$-dimensional set $D_r(t)={\RR_t} \cap\pd I_r$.
This is formalized in the appendix (see Theorem \ref{Maggi})
and used in the proof of Lemma \ref{flux-main} below.

We will need the following isoperimetric inequalities. 

The Euclidean Isoperimetric Inequality (Theorem 14.1 in~\cite{M})
implies that the volume $\rci(\infty,t)=|\RR_t|$ of the entire reachable set
$\RR_t$ and its perimeter $s(\infty,t)$ satisfy
\be\label{isoperimeter}
 s(\infty,t) \ge \lambda_0\, \rci(\infty,t)^{\frac{\size-1}\size},
\ee
where $\lambda_0=\size \vol_{\size}^{1/{\size}}$ is the Euclidean isoperimetric constant
satisfying
$$
  |\pd B_r(0)| = \lambda_0\,|B_r(0)|^{\frac{\size-1}\size} \hbox{ for all } r>0.
$$

The Relative Isoperimetric Inequality in the cube (Theorem~\ref{iso} in Appendix)
implies that the volume $\rci(r,t)$ of $\RR_t\cap I_r$ and its relative perimeter
$s(r,t)$ inside $I_r$ satisfy
\be\label{isopercube}
  s(r,t) \ge \lambda_1 \big(\min\{ \rci(r,t), |I_r|-\rci(r,t) \}\big)^{\frac{\size-1}\size} ,
\ee
where $\lambda_1$ is a positive constant depending only on~$\size$.

%


%

\section{Proof of Theorem~\ref{newmain} and Corollary~\ref{new_cor}}\label{stages}

Most of this section we spend proving Theorem \ref{newmain}.
Its most technical stage
(namely the proof of Proposition \ref{sequence}) is put off.
It is contained in Sections \ref{leak} and~\ref{middle}.

Let us say a few words about how the proof of Theorem~\ref{newmain} goes. 
It is easy to show that the volume of $\RR_t$ grows to infinity.
It is a more delicate task to verify that the set $\RR_t$ cannot be carried away from the origin by the flow.
Our idea is to show that,
for every $r\geq 0$, the set $\RR_t \cap I_r$ fills $I_r$ for all sufficiently large $t$. 
Thus we look at how the volume $\rci(r, t)=|\RR_t\cap I_r|$ grows. 
We want it to reach $(2r)^{\size}$, the volume of $I_{r}$.
This is done by dividing the filling process into three stages. During the initial stage we fill
in at least $\alpha|I_r|$ of the volume of $I_r$,
where $\alpha$ is a small positive constant defined below. 
In the next step, which is the key one, we 
fill in at least  $(1-\alpha)|I_r|$ of the volume of~$I_r$.
Furthermore, this portion of volume 
remains filled forever after a certain time~$t$. 
Finally, we show that at a later time a smaller cube $I_{r/2}$ is completely filled.
Since the choice of $r$ is arbitrary, $r/2$ is as good as $r$.

Our choice of $\alpha$ depends on the maximal speed of the fluid flow and the dimension. We fix
\be
\label{def-alpha}
  \alpha = \frac{\vol_{\size}}{(4M)^{\size}}
\ee
for the rest of the proof. We assume that $r$ is sufficiently large,
more precisely $r\ge r_0$ where $r_0$ is a constant
depending on $V_t$. 
The precise value of $r_0$ is defined in the course of the proof.

\medskip

The initial stage of the filling process is simple.
It is analyzed  in the following lemma:

\begin{lem}
\label{start}
Let $r>0$ and $T_0=\frac r{2M}$. Then
$$
 \rci(r,T_0) \ge \alpha |I_r| .
$$
\end{lem}

\begin{proof}
By \eqref{rGrowth_short} we have $ \RR_{T_0} \subset I_r $,
hence $\rci(r,T_0)=| \RR_{T_0}|$.
Clearly $\RR_t$ has a nonempty interior and hence $|\RR_t|>0$ for every $t>0$.
By \eqref{vGrowth-infty} and the isoperimetric inequality \eqref{isoperimeter}
we have
$$
  \frac d{dt}|\RR_t| \ge s(\infty,t) \ge  \size \vol_{\size}^{1/{\size}} |\RR_t|^{\frac{\size-1}\size} .
$$
Therefore
\be\label{Growth_incompress2}
  |\RR_t| \ge \vol_{\size}\cdot  t^{\size} = |B_t(0)| .
\ee
Hence 
$
  \rci(r,T_0) \ge \vol_{\size} T_0^{\size} = \vol_{\size} (2M)^{-\size}\, r^{\size} = \alpha |I_r|
$.
\end{proof}

The middle stage of the filling process is the most technical.
This is the content of the next proposition.

\begin{prop}
\label{sequence}
There exist constants $A=A(n)\ge 1$
and $r_0>0$ such that
$\rci(r,t)> (1-\alpha) |I_r| $ for all $r\ge r_0$ and $t\ge A r$.
\end{prop}

We prove Proposition~\ref{sequence} in Section~\ref{middle}.
For this proof we need to estimate how much volume of $\RR_t\cap I_r$ can leak out through
the boundary of $I_r$. 
This estimate is contained in Section \ref{leak}, see Proposition~\ref{lemma5}.

The final stage of the filling process is simple again. 
It is analyzed  in Lemma \ref{finish}.
We show that, once $\rci(r,t)$
exceeds $(1-\alpha)|I_r|$, then in time $T_0$ the reachable
set covers the smaller cube $I_{r/2}$.

\begin{lem}
\label{finish}
Suppose that $r>0$ and $t_1>0$ are such that $\rci(r,t_1)>(1-\alpha)|I_r|$.
As in the previous lemma, let $T_0=\frac{r}{2M}$. Then
$$
I_{r/2}\subset\RR_{t_1+T_0}.
$$
\end{lem}

\begin{proof}  
Fix $p\in I_{r/2}$ and let $t_2=t_1+T_0$.  Let 
$$
  \RR^-_t=\{x\in \R^{\size}: \hbox{$(p,t_2)$ is reachable from $(x,t_2-t)$}\}.
$$
$\RR^-_t$ is the  reachable set from $p$ for the reversed flow $V^-_{t}=-V_{t_2-t}.$
As in the previous lemma we can apply~\eqref{Growth_incompress2} to $V^-_t$ to obtain
$$
  |\RR^-_{T_0}|\geq \vol_{\size} T_0^{\size}  = \alpha |I_r|. 
$$
Hence
$$
  |\RR^-_{T_0}|+\rci(r,t_1)>|I_r|.
$$
By~\eqref{rGrowth_short}  applied to $V_t^-$ we have
$$
  \RR^-_{T_0}\subset B_{r/2}(p)\subset I_r .
$$
Thus $\RR^-_{T_0}\cap\RR_{t_1} \neq \emptyset$.  Hence $p\in\RR_{t_2}.$
\end{proof}

Combining the results of the three stages, we obtain the following
proposition, which is essentially Theorem \ref{newmain}
with non-optimal bounds on reach time.

\begin{prop}
\label{newmain-p}
There exist constants $\mu=\mu(\size)\in(0,1)$ and $C>0$ such that
for every  $t\ge C$ we have
$
  B_{\mu t}(0) \subset \RR_t
$.
\end{prop}

\begin{proof}
By Proposition \ref{sequence} we have $\rci(r,t) > (1-\alpha)|I_r|$
for all $r\ge r_0$ and $t\ge Ar$.
By Lemma \ref{finish} it follows that
$$
  B_{r/2}(0) \subset I_{r/2} \subset \RR_t
$$
for all $t\ge Ar+T_0=(A+\frac1{2M})r$.
Applying this to $2r$ in place of $r$ yields that
$B_r(0) \subset \RR_t$
for all $r\ge r_0$ and $t\ge (2A+1) r$.
Hence the statement 
holds for $\mu=(2A+1)^{-1}$ and $C=(2A+1)r_0$.
\end{proof}

Now we are in a position to prove Theorem~\ref{newmain} and Corollary~\ref{new_cor}.

\begin{proof}[Proof of Theorem~\ref{newmain}]
Fix $\ep>0$. 
Note that Proposition \ref{newmain-p} 
(after a suitable rescaling)
holds for controls bounded
by $\ep$ instead of~1.
Our plan is to spare a small part of control
to ensure reachability and use the remaining part of control
to add the drift with speed $1-\ep$ in a desired direction.

Without loss of generality assume
that $x_0=0$ and $t_0=0$.
Fix $v\in\R^n$ such that $|v|\le 1-\ep$ and
apply Proposition \ref{newmain-p} to the flow
$\widetilde V$ defined by
$$
 \widetilde V_t(x) = \frac1\ep\, V_t(\ep x  +  tv) .
$$
This yields a constant $C_{\ep,v}>0$ such that
for every $t\ge C_{\ep,v}$ the reachable set for $\widetilde V$
at time $t$ contains the ball $B_{\mu t}(0)$.
Here $\mu=\mu(n)$ is the constant from Proposition \ref{newmain-p}.
If $\widetilde\gamma:[0,t]\to\R^n$
is an admissible path for $\widetilde V$,
then the path $\gamma$ defined by
$$
 \gamma(\tau) = \ep\widetilde\gamma(\tau)+\tau v
$$
is admissible for our flow $V$.
Hence the reachable set $\RR_t$
contains the ball $B_{\ep\mu t}(tv)$.
In particular the point $y=tv$ can be reached
at time $t$, which satisfies $t\le |y|/(1-\ep)$.

It remains to show that the constant $C_{\ep,v}$
can be chosen independently of $v$. 
To show this, let us choose a finite $\ep\mu$-net
$\{v_1,\dots,v_m\}$ in the ball $B_{1-\ep}(0)$
and let $C_\ep=\max\{C_{\ep,v_i}: 1\le i\le m\}$.
Then for every $t\ge C_\ep$ we have
$$
  \RR_t \supset \bigcup_{i=1}^m B_{\ep\mu t}(tv_i) \supset B_{(1-\ep)t}(0) .
$$
Thus every point $x\in\R^n$ is reachable
at any moment $t\ge\max\{C_\ep,|x|/(1-\ep)\}$.
To finish the proof of the theorem,
set $\ep=1-\frac1a$ and $C=C_\ep$.
\end{proof}

\begin{proof}[Proof of Corollary~\ref{new_cor}]
The idea is to use a part of the control
to compensate the mean drift at some scale.
Fix $c\in(\Delta,1)$.
By \eqref{dos_sod} there exists $L_0>0$ such that the flow 
\be\label{av_L0}
 \overline V_t(x) := \frac{1}{L_0^{\size}} \int_{[0,L_0]^{\size}} V_t(x+y) \, dy
\ee
satisfies $\|\overline V_t(x)\| < c$ for all $t$ and~$x$.
Let $V_t^0=V_t-\overline V_t$. Then
\be\label{av_av}
 \left\| \frac{1}{L^{\size}} \int_{[0,L]^{\size}} V^0_t(x+y) \, dy \right\| 
 \le \frac{2^nL_0M}{L}
\ee
for all $L\ge L_0$.
Indeed, let $\Phi_L$ denote the characteristic function of the cube $[-L,0]^n$ divided by~$L^n$.
Then $\overline V_t$ is the convolution $V_t*\Phi_{L_0}$
and the integral in \eqref{av_av} is the value at $x$ of the convolution
$V^0_t*\Phi_L=V_t*(\Phi_L-\Phi_{L_0}*\Phi_L)$.
The function $|\Phi_L-\Phi_{L_0}*\Phi_L|$ is bounded by $1/L^n$ and its support is contained
in the set $[-L-L_0,0]^n\setminus[-L,-L_0]^n$ of volume
$(L+L_0)^n-(L-L_0)^n \le 2^nL_0L^{n-1}$.
Hence the $L^1$-norm of $\Phi_L-\Phi_{L_0}*\Phi_L$ is bounded by $2^nL_0/L$
and \eqref{av_av} follows.

Observe that $\overline V_t$ is incompressible and bounded by~$M$.
This and \eqref{av_av} imply that $V^0_t$ satisfies the assumptions of Theorem~\ref{newmain}.
We apply Theorem~\ref{newmain} to $V^0_t$ with the maximal fish speed set to
$1-c$ instead of~1.
Since $\|\overline V_t\|<c$, every admissible path in this setting
is admissible for the original flow $V_t=V_t^0+\overline V_t$.
Because of the speed renormalization, the conclusion of the theorem holds for any $a>\frac1{1-c}$.
Since $c\in(\Delta,1)$ is arbitrary, Corollary~\ref{new_cor} follows.
\end{proof}

\begin{rmk}\label{newcor_uniform}
One can see from the proof that the constant $C$ in Corollary \ref{new_cor}
is determined by $M$, $\Delta$, $a$, and any value $L_0$ such that $\overline V_t(x)$
in \eqref{av_L0} is bounded by $\frac{\Delta+1}2$ for all $t$ and~$x$.
\end{rmk}

\section{Volume change estimate}\label{leak}

Throughout the paper we integrate areas and perimeters over time intervals. 
Such integrals are indicated by a hat. Namely we define
$$
\hat s(r,t,T)=\int_{t}^{t+T}s(r,\tau)\,d\tau 
$$
and
$$
\hat p(r,t,T)=\int_{t}^{t+T} p(r,\tau)\, d\tau .
$$

The goal of this section is to prove the following proposition.

\begin{prop}\label{lemma5}
For every $\ep>0$ there exists $r_0>0$ 
such that for all $r\ge r_0$, $t>0$, $T\in[0,r]$ we have
\be\label{lemma5-eq}
  \rci(r,t+T)-\rci(r,t) \ge \hat s(r,t,T) -\ep r^{\size} .
\ee
\end{prop}


For the proof of Proposition~\ref{lemma5} we need the following two lemmas.

\begin{lem}\label{lemma3}
For all $r,t,T>0$,
\be\label{hatS-above}
  \hat s(r,t,T) \le C_1(r+T) r^{\size-1}
\ee
where $C_1=\size 2^{\size}M$.
\end{lem}

\begin{proof}
From Lemma \ref{lemma_fact1} and a trivial estimate
$$
 |\flux(V_t,D_r(t))| \le M |\pd I_r|
$$
we have
$$
  \frac d{dt} \rci(r,t) \ge s(r,t) - M |\pd I_r|
$$
(in the sense of distributions).
By integrating this we obtain
$$
  \rci(r,t+T)-\rci(r,t) \ge \hat s(r,t,T) - M T |\pd I_r| .
$$
The left-hand side is bounded above by $|I_r|$. Hence
$$
 \hat s(r,t,T) \le |I_r|  +  M  T  |\pd I_r| .
$$
Since $|I_r|=2^nr^n$ and $|\pd I_r|=n2^n r^{n-1}$, \eqref{hatS-above} follows.
\end{proof}

The incompressibility and small mean drift assumptions
imply the following lemma, which we borrow from \cite{BIN}.
This is the only place in the proof where the small mean drift
assumption is used.

\begin{lem}[cf.~{\cite[Lemma 3.1]{BIN}}]\label{Aepsilon}
For every $\ep>0$ there exists $L_0>0$ such that the following holds.
Let $F$ be an $(n-1)$-dimensional cube with edge length $L\ge L_0$,
then
\begin{equation}\label{dosbissmall}
| \flux(V,F)| \le \ep L^{n-1}.
\end{equation}
\end{lem}

\begin{proof} 
This lemma is stated in \cite{BIN} for a time-independent vector field.
We apply \cite[Lemma 3.1]{BIN} to the vector field $V_t$
for every fixed~$t$. 
The constant $L_0$ (named $A_0$ in \cite[Lemma 3.1]{BIN})
depends on the vector field, so we need to make sure that
it can be chosen independently of $t$.
In the proof in~\cite{BIN} one can see that $L_0$
depends only on $M$ and on the rate of convergence of
the mean drift to~0. 
Hence the proof works for our Lemma \ref{Aepsilon} as well.
\end{proof}

\begin{lem}[cf.~{\cite[Lemma 3.3]{BIN}}]\label{flux-main}
For every $\ep>0$ there exist $r_1>0$ and $C_0>0$ such that for 
almost all $t>0$ and $r\ge r_1$,
\be\label{flux-mixed}
 |\flux(V_t,D_r(t))| \le C_0 p(r,t)+\ep r^{\size-1} .
\ee
\end{lem}

\begin{proof}
This lemma could also be borrowed from \cite{BIN} if we had proven
certain regularity properties of $\RR_t$.
For the sake of completeness we include a proof here.
The proof is essentially the same but it is based on different
foundations in Geometric Measure Theory.

We fix $\ep>0$ and apply Lemma \ref{Aepsilon}.
Let $L_0$ be the constant provided by Lemma \ref{Aepsilon}.
Let $t>0$ be such that $\RR_t$ has finite perimeter.
Assume that $r\ge L_0$ and the following holds:
For every hyperplane $\Sigma$ containing one of the 
$(n-1)$-dimensional faces
of the cube $I_r$, the slice $\RR_t\cap\Sigma$ has finite perimeter
in $\Sigma\cong\R^{n-1}$ and its reduced boundary
in $\Sigma$ coincides with $\Sigma\cap\pd^*\RR_t$ up to a set
of zero $(n-2)$-dimensional Hausdorff measure.
By the Boundary Slicing Theorem \ref{Maggi}
these conditions are satisfied for almost all $r$.

Since $r\ge L_0$, we have $r=mL$ for some $L\in[L_0,2L_0]$ and $m\in\Z$.
We divide $\pd I_r$ into $(n-1)$-dimensional cubes $F_i$, $i=1,2,\dots,2nm^{n-1}$,
with edge length $L$.
Denote $D=D_r(t)$ for brevity.
For each $i$, define 
$$
s_i = \min\{|F_i\cap \RR_t|, |F_i\setminus\RR_t|\}
= \min\{|F_i\cap D|, |F_i\setminus D|\}
$$
and
$$
 p_i = P_{n-1}(D,F_i^\circ) = \HA^{n-2}(F_i^\circ\cap\pd^*\RR_t)
$$
where $P_{n-1}$ denotes the perimeter in the respective hyperplane
and $F_i^\circ$ is the relative interior of~$F_i$.
The last identity follows from the De Giorgi Theorem~\ref{dg}.

The isoperimetric inequality in
$(n-1)$-dimensional cubes
implies  that
$$
 s_i \le CL p_i
$$
where $C$ is a constant depending only on $\size$. 
For $\size \geq 3$, we prove this  isoperimetric inequality in Appendix, Corollary~\ref{isopercube-lite}. For $\size=2$ Corollary~\ref{isopercube-lite} is trivially true.
Therefore, we have
$$
 \bigl| |\flux(V_t,F_i\cap D)|-|\flux(V_t,F_i\setminus D)| \bigr| 
 \le |\flux(V_t,F_i)| \le \ep L^{n-1},
$$
where the second inequality follows from Lemma \ref{Aepsilon}.
At least one of the quantities $|\flux(V_t,F_i\cap D)|$
and $|\flux(V_t,F_i\setminus D)|$ is bounded by $Ms_i$,
hence both of them are bounded by $Ms_i+\ep L^{n-1}$.
Thus
$$
 |\flux(V_t,F_i\cap D)| \le M s_i + \ep L^{n-1} \le CMLp_i + \ep L^{n-1}
 \le C_0p_i + \ep |F_i|,
$$
where $C_0=2CML_0$.
Summing up over all $i$ yields that
$$
 |\flux(V_t,D)| 
 \le C_0 \sum p_i + \ep |\pd I_r|
 \le C_0 p(r,t)+n2^n \ep r^{n-1}
$$
for almost all $r\ge L_0$.
Since $\ep$ is arbitrary, the lemma follows.
\end{proof}

\begin{proof}[Proof of Proposition~\ref{lemma5}]
Fix $\beta,\ep>0$.
We apply Lemma \ref{flux-main} to $\ep_1:=\ep/2^{n+1}$ in place of $\ep$.
This yields
$$
|\flux(V_t,D_r(t))| \le C_0 p(r,t) +\ep_1 r^{\size-1}
$$
for almost all $r\ge r_1$ and $t>0$. This and \eqref{vGrowth} imply
$$
  \frac{d}{dt} \rci(r,t) \ge s(r,t) - C_0 p(r,t) - \ep_1 r^{\size-1} 
$$
for almost all $r>r_1$ and $t>0$. 
Integration in $t$ yields
\be\label{long-estimate}
 \rci(r,t+T)-\rci(r,t) \ge \hat s(r,t,T) - C_0 \hat p(r,t,T) - \ep_1 T r^{\size-1} 
\ee
for almost all $r>r_1$ and all $t, T>0$.

Define
$$
  h  := \frac { 2^{\size+1} C_0C_1 } {\ep_1} ,
$$
where $C_1$ is the constant from Lemma \ref{lemma3}.
By the co-area inequality~\eqref{co-area},
$$
 s (r+h,t) \ge \int_0^{r+h} p(x,t)\,dx \ge \int_r^{r+h} p(x,t)\,dx
$$
for all $r>0$ and almost all $t>0$.
Once again, integration in $t$ yields
\be\label{int-hatp}
 \hat s(r+h,t,T) \ge \int_r^{r+h} \hat p(x,t,T)\,dx
\ee
for all $r>0$ and all $t,T>0$.

Now let $r$ and $t$ be as in the formulation of Proposition \ref{lemma5}.
Namely $t>0$ is arbitrary, $r\ge r_0$ where $r_0$ is to be chosen later, and $0\le T\le r$.
We require that $r_0\ge r_1$ and $r_0\ge h$, the latter ensures that $h\le r$.
By Lemma \ref{lemma3} applied to $r+h$ in place of $r$,
$$
 \hat s(r+h,t,T) \le C_1(r+h+T) (r+h)^{\size-1}
 \le 2^{\size+1} C_1 r^{\size}
$$
since $T\le r$ and $h\le r$.
This and \eqref{int-hatp} imply that
there exists $\tilde r\in[r,r+h]$ such that
\be\label{bigbox}
  \hat p(\tilde r,t,T) \le \frac { 2^{\size+1} C_1 r^{\size}   }  h
  = C_0^{-1}\ep_1 r^{\size},
\ee
where the equality follows from the definition of $h$.
Furthermore the set of  $\tilde r\in[r,r+h]$ satisfying \eqref{bigbox}
has positive measure,
hence we can choose $\tilde r$ so that \eqref{bigbox} holds and 
\eqref{long-estimate} applies to $\tilde r$ in place of~$r$:
$$
 \rci(\tilde r,t+T)-\rci(\tilde r,t) 
 \ge \hat s(\tilde r,t,T) - C_0 \hat p(\tilde r,t,T) - \ep_1 T \tilde r^{\size-1} .
$$
This estimate, 
\eqref{bigbox}, and the inequalities $T\le r$ and $\tilde r\le 2r$ imply that
$$
 \rci(\tilde r,t+T)-\rci(\tilde r,t) 
 \ge \hat s(\tilde r,t,T) - \ep_1 r^n - 2^{n-1}\ep_1 r^n
 \ge \hat s(\tilde r,t,T) - 2^n\ep_1 r^n = \hat s(\tilde r,t,T) - \tfrac12 \ep r^n .
$$
Since $\tilde r\ge r$, we have $\hat s(\tilde r,t,T)\ge \hat s(r,t,T)$.
Thus
\be\label{almost-done}
 \rci(\tilde r,t+T)-\rci(\tilde r,t)
 \ge   \hat s(r,t,T)-\tfrac12 \ep r^n.
\ee

Now we estimate the difference between $\rci(\tilde r,t+T)$ and $\rci(r,t+T)$:
$$
  \rci(\tilde r,t+T) - \rci(r,t+T) = | \RR_{t+T} \cap (I_{\tilde r}\setminus I_r)|
  \le |I_{\tilde r}\setminus I_r| = 2^n (\tilde r^n-r^n) .
$$
The right-hand side is bounded as follows:
$$
 2^n (\tilde r^n-r^n) \le n 2^n (\tilde r - r) \tilde r^{n-1} 
 \le n 2^n h \tilde r^{n-1} \le n 2^{2n-1} h r^{n-1} \le \tfrac12 \ep r^n
$$
if we require that 
\be\label{rzero}
r\ge r_0\ge  n 2^{2n-1} h \ep^{-1}.
\ee
Thus
\be\label{cube-difference}
  \rci(\tilde r,t+T) - \rci(r,t+T) \le \tfrac12\ep r^{\size} .
\ee
This and a trivial inequality $\rci(r,t)\le\rci(\tilde r,t)$
imply that
$$
  \rci(r,t+T)-\rci(r,t) \ge \rci(\tilde r,t+T)-\rci(\tilde r,t) - \tfrac\ep2 r^{\size} 
  \ge \hat s(r,t,T)-\ep r^{\size} ,
$$
where the second inequality follows from \eqref{almost-done}.
This finishes the proof of Proposition \ref{lemma5}.
\end{proof}

\section{Middle stage. Proof of Proposition~\ref{sequence}}\label{middle} 

In this section we prove Proposition \ref{sequence},
the last remaining piece of the proof of Theorem \ref{newmain}.
The proof is based on Proposition \ref{lemma5}
and the isoperimetric inequality \eqref{isopercube}
for subsets of a cube.

\medskip

To facilitate understanding of the proof, we first give its simplified
version assuming that the estimate \eqref{lemma5-eq} from Proposition \ref{lemma5}
holds without the correction term $-\ep r^n$. After this simplification
the estimate \eqref{lemma5-eq} boils down to the differential inequality
\be\label{simple}
  \frac d{dt} \rci(r,t) \ge s(r,t) \ge \lambda_1 \min \{ \rci(r,t), |I_r|-\rci(r,t) \}^{\frac{n-1}n}
\ee
where the second inequality is the isoperimetric inequality \eqref{isopercube}.
This implies that $\rci(r,t) \ge \phi(t)$ where $\phi(t)>0$ solves the ODE
$$
  \frac d{dt} \phi(t) = \lambda_1 \min \{ \phi(t), |I_r|-\phi(t) \}^{\frac{n-1}n}
$$
with the initial condition $\lim\limits_{t\to 0+} \phi(t)=0$.
The solution is given by
$$
  \phi(t) = 
  \begin{cases}
    a t^n, &\quad t\in [0, b] , \\
    |I_r|-a(2b-t)^n, &\quad t\in[b,2b] ,
  \end{cases}
$$
where $a=(\lambda_1/n)^n$, and
$b=(\frac1{2a}|I_r|)^{1/n}=cr$ with $c=2^{\frac{n-1}n} n \lambda_1^{-1}$.
It reaches the value $\phi(t)=|I_r|$ at $t=2b=2cr$, and 
the coefficient $2c$ depends only on $n$.
This proves the main theorem under the above simplifying assumption.

The actual proof of Proposition \ref{sequence} is essentially 
a discrete version of the above argument.
We apply Proposition \ref{lemma5} to $T=\beta r$
where $\beta\in(0,1)$ is a carefully chosen constant
(depending on the flow but not depending on $r$).
This yields a lower bound for $\rci(r,T_k)$ 
where $T_k=T_0+k\beta r$, $k=1,2,\dots$.
It turns out that for a sufficiently small $\ep>0$
the term $\hat s(r,t)$ dominates the correction term $-\ep r^n$
and hence the resulting bound for $\rci(r,T_k)$
is similar to the formula for $\phi(T_k)$. This implies the desired conclusion.

Another technical issue is that the isoperimetric inequality \eqref{isopercube}
does not integrate well over time intervals.
This is handled in Lemma \ref{lemma6} below,
where we prove a discrete analogue of the differential inequality \eqref{simple}.

\medskip

Now we are back to the formal proof.
Recall that we have a fixed $\alpha$ defined by \eqref{def-alpha}. 
We now choose a small constant $\beta\in(0,1)$.
First we require that $\beta<\frac{\alpha}{10}$.
%
%
Second, we require that $\beta$ is so small that the following holds.
For all $x\in[\frac\alpha2,1]$ and all $\delta\in[0,\beta]$
\be\label{beta-req2}
 (x+\delta)^{1/{\size}} - x^{1/{\size}} \ge \frac1{2{\size}} x^{\frac{1-\size}\size} \delta .
\ee
Such $\beta$ exists since the function $x\mapsto x^{1/{\size}}$ is smooth
on $[\frac\alpha2,1]$ and its derivative equals $\frac{1}{\size} x^{\frac{1-\size}\size}$.

We fix $\alpha$ and $\beta$ for the rest of the proof.

\begin{lem}\label{lemma6}
There exist $\lambda=\lambda(\size)\in(0,1]$ and $r_0>0$ 
such that for every $r\ge r_0$ and $T=\beta r$
the following holds.

1. For all $t>0$ and $\tau\in[t,t+T]$,
\be\label{alpha/10}
 \rci(r,\tau) \ge \rci(r,t) - \tfrac\alpha{10} |I_r| .
\ee

2. If $t>0$ satisfies 
\be\label{away-from01}
  \tfrac\alpha2|I_r| \le \rci(r,t) \le (1-\tfrac\alpha2)|I_r| ,
\ee
then
\be\label{lemma6-eq}
 \rci(r,t+T) \ge \rci(r,t) +\lambda T m(t)^{\frac{\size-1}\size}
\ee
where
$$
 m(t)= \min\{ \rci(r,t), |I_r|-\rci(r,t)\}  . 
$$
\end{lem}

\begin{proof}
Fix a sufficiently small $\ep>0$, namely
$$
  \ep < \min \{ \tfrac{\alpha}{10} , \tfrac1{16}\lambda_1\alpha\beta \} ,
$$
where $\lambda_1=\lambda_1(n)$ is the isoperimetric constant from \eqref{isopercube}.
By Proposition \ref{lemma5} there exists $r_0>0$ such that
\be\label{l5-eq}
 \rci(r, \tau)-\rci(r,t) \ge \hat s(r,t, \tau) - \ep r^{\size}
\ee
for any $r\ge r_0$, $T=\beta r$, and  $\tau\in[t,t+T]$.
Since $\hat s(r,t,\tau)\ge 0$, this implies that
$$
 \rci(r, \tau) - \rci(r,t) \ge - \ep r^{\size} > - \tfrac{\alpha}{10} |I_r|
$$
due to the choice of $\ep$. This proves the first claim of the lemma.

To prove the second one,
define 
$$
m_0 = \inf \{ m(\tau) : \tau\in[t,t+T] \}
$$
and consider two cases: $m_0<\frac12 m(t)$ and $m_0\ge \frac12 m(t)$.

\medskip

\textit{Case 1: $m_0<\frac12 m(t)$}.
Then $m(\tau)<\frac12 m(t)$ for some $\tau\in[t,t+T]$.
The definition of $m(t)$ and \eqref{away-from01}
imply that 
\be\label{subcase-a}
 |\rci(r,\tau) - \rci(r,t)| > \tfrac\alpha4|I_r| .
\ee
The inequality \eqref{alpha/10} rules out the case $\rci(r,\tau)<\rci(r,t)$,
 hence
$$
 \rci(r,\tau) > \rci(r,t) + \tfrac\alpha4|I_r| .
$$
Combining this inequality with \eqref{alpha/10} applied to $\tau$ and $t+T$ in place of $t$ and $\tau$, respectively, yields
\be\label{3/10}
 \rci(r,t+T) \ge \rci(r,\tau)-\tfrac\alpha{10}|I_r|
 > \rci(r,t)+\tfrac{\alpha}{10}|I_r| .
\ee
On the other hand, by the trivial estimate $m(t)\le|I_r|=(2r)^{\size}$ we have
$$
 Tm(t)^{\frac{\size-1}\size } \le T (2r)^{\size-1} = \beta r (2r)^{\size-1} = \tfrac\beta2|I_r| < \tfrac{\alpha}{10}|I_r| .
$$
This and \eqref{3/10} imply \eqref{lemma6-eq} for any $\lambda\le 1$.

\medskip

\textit{Case 2: $m_0\ge\frac12 m(t)$}.
By the isoperimetric inequality \eqref{isopercube} for subsets of the cube,
$$
 s(r,\tau) \ge \lambda_1 m(\tau)^{\frac{\size-1}\size} 
 \ge \lambda_1 m_0^{\frac{\size-1}\size} \ge \tfrac12\lambda_1 m(t)^{\frac{\size-1}\size}
$$
for all $\tau\in[t,t+T]$. Hence
\be\label{isoper}
  \hat s(r,t,T) \ge \tfrac12\lambda_1 T m(t)^{\frac{\size-1}\size} .
\ee
By \eqref{away-from01}, we have $m(t)\ge\frac\alpha2|I_r| = \frac\alpha2(2r)^{\size}$.
Therefore
\be\label{1/8}
 \hat s(r,t,T) \ge \tfrac12\lambda_1 T m(t)^{\frac{\size-1}\size}
 = \tfrac12\lambda_1 \beta r m(t)^{\frac{\size-1}\size} 
 \ge \tfrac14 \lambda_1 \beta r \big(\tfrac\alpha2(2r)^n\big)^{\frac{\size-1}\size} 
 > \tfrac18\lambda_1\alpha\beta r^n > 2\ep r^n,
\ee
where the last inequality follows from the choice of $\ep$.
Inequalities \eqref{1/8}, \eqref{l5-eq} and \eqref{isoper} imply that
\be\label{lambda-final}
 \rci(r,t+T)-\rci(r,t) 
 \ge \hat s(r,t,T) - \ep r^n
 \ge \tfrac12 \hat s(r,t,T) 
 \ge \tfrac14 \lambda_1 T m(t)^{\frac{\size-1}\size} .
\ee
The inequality \eqref{lambda-final} implies \eqref{lemma6-eq} for $\lambda=\tfrac14\lambda_1$.

Combining the outcomes of the two cases, one sees that \eqref{lemma6-eq}
holds for $\lambda=\min\{1,\frac14\lambda_1\}$.
\end{proof}

Now we are in a position to prove Proposition \ref{sequence}.
The proof is a straightforward but technical implication
of Lemma \ref{lemma6}.
Nothing beyond basic analysis is used.

\begin{proof}[Proof of Proposition~\ref{sequence}]
Let $r_0$ be such that the assertion of Lemma \ref{lemma6} holds.
Fix $r>r_0$ and define a function $f:\R_+\to[0,1]$ by
$$
  f(t) := \frac{\rci(r,t)}{|I_r|} = \frac{\rci(r,t)}{(2r)^{\size}} .
$$
We rewrite some of the previous results in terms of $f$.
First, Lemma \ref{start} turns into the inequality
\be\label{f1}
  f(T_0) \ge \alpha \qquad\text{where}\qquad T_0 = \tfrac r{2M} .
\ee
By the first statement of Lemma \ref{lemma6} we have
\be\label{f2}
 f(\tau) \ge f(t) -\tfrac\alpha{10} \qquad\text{if}\qquad t\le\tau\le t+\beta r .
\ee
Finally, the second statement of Lemma \ref{lemma6} takes the form:
\be\label{f3}
 f(t+\beta r) \ge f(t) + \tfrac12 \lambda \beta \min\{f(t),1-f(t)\}^{\frac{\size-1}\size}
 \quad\text{provided that}\quad \tfrac\alpha2\le f(t)\le 1-\tfrac\alpha2 .
\ee
Here $\lambda=\lambda(\size)\in(0,1]$ is the constant from Lemma \ref{lemma6},
and we use this notation throughout the rest of the proof.

In our new notation the statement of Proposition~\ref{sequence} turns into
$$
 f(t)> 1-\alpha \quad\hbox{for all}\quad t\ge A r
$$
where $A$ is a constant depending only on $n$.

Now consider a sequence $\{y_k\}_{k=0}^\infty$ defined by
$y_k=f(T_0+k\beta r)$.
The relations \eqref{f1}--\eqref{f3} imply the following properties
of this sequence:
\begin{enumerate}
\item $y_0\ge\alpha$~\label{oneP};
\item if $\frac\alpha2 \le y_k \le \frac12$ then $y_{k+1} \ge y_k+\frac12\lambda\beta y_k^{\frac{\size-1}\size}$~\label{twoP};
\item if $\frac12\le y_k \le 1-\frac\alpha2$ then $y_{k+1} \ge y_k+\frac12\lambda\beta(1-y_k)^{\frac{\size-1}\size}$~\label{threeP};
\item if $y_k\ge 1-\frac\alpha2$ then $y_{k+1} \ge 1-\frac6{10}\alpha$~\label{fourP}.
\end{enumerate}
It follows that $\{y_k\}$ increases as long as it stays below $1-\frac\alpha2$,
and if it gets above $1-\frac\alpha2$ then
after that it is confined to the interval $[1-\frac6{10}\alpha,1]$.
We are going to prove that $y_k$ eventually attains a value greater than $1-\frac\alpha2$,
and estimate the index $k$ for which this happens.

If $y_k\le\frac12$ then by~\eqref{oneP} and~\eqref{twoP}
we have $y_k\ge y_0\ge\alpha$ and
$ y_{k+1} \ge y_k+\delta_k$
where $\delta_k =\frac12\lambda\beta y_k^{\frac{\size-1}\size}$.
Hence
$$
 y_{k+1}^{1/{\size}} \ge (y_k+\delta_k)^{1/{\size}} \ge y_k^{1/{\size}} + \frac1{2{\size}} y_k^{\frac{1-\size}\size}\delta_k 
 = y_k^{1/{\size}} + \frac{\lambda\beta}{4{\size}} .
$$
Here the second inequality follows from the choice of $\beta$
(see \eqref{beta-req2}) and the fact that $\delta_k\le\beta$ since $\lambda\le 1$ and $y_k\le 1$.
By induction it follows that
$$
  y_k^{1/{\size}} \ge y_0^{1/n} + \frac{\lambda\beta k}{4{\size}} > \frac{\lambda\beta k}{4{\size}}
$$
as long as $y_0,\dots,y_{k-1}\le \frac12$.
Hence there exists $k_1\le \frac{4{\size}}{\lambda\beta}$ such that $y_{k_1} \ge\frac12$.

Now consider $k\ge k_1$. Note that $y_k\ge\frac12$ by \eqref{twoP}--\eqref{fourP}.
As long as $y_k\le 1-\frac\alpha2$, we have
\be
\label{yn1}
 y_{k+1} \ge y_k + \delta_k,
\ee
where
$$
  \delta_k =\tfrac12\lambda\beta (1-y_k)^{\frac{\size-1}\size} .
$$
We rewrite \eqref{yn1} as follows:
$$
 (1-y_{k+1})^{1/{\size}} \le (1-y_k-\delta_k)^{1/{\size}} 
 \le (1-y_k)^{1/{\size}} - \frac1n (1-y_k)^{\frac{1-\size}\size} \delta_k
 =  (1-y_k)^{1/{\size}} - \frac{\lambda\beta}{2{\size}} .
$$
Here the second inequality follows from the concavity of the function $t\mapsto t^{1/{\size}}$.
By induction it follows that
$$
 (1-y_k)^{1/{\size}} 
 \le (1-y_{k_1})^{1/{\size}} -\frac{\lambda\beta}{2{\size}} (k-k_1)
 \le 1 - \frac{\lambda\beta}{2{\size}} (k-k_1)
$$
as long as $y_{k_1},\dots, y_{k-1}\le 1-\frac\alpha2$.
Hence there exists $k_2 \le k_1+\frac{2{\size}}{\lambda\beta} \le \frac{6{\size}}{\lambda\beta}$
such that $y_{k_2} \ge 1-\frac\alpha2$.
Then~\eqref{threeP} and~\eqref{fourP} imply that $y_k\ge 1-\frac6{10}\alpha$ for all $k\ge k_2$.

This and \eqref{f2} imply that $f(t)\ge 1-\frac7{10}\alpha$ for all 
$t\ge T_0+\beta r k_2$.
Since  $T_0+\beta r k_2 \le T_0 + \frac{6{\size}}\lambda r \le (\frac{6{\size}}\lambda+1) r$,
the statement of Proposition \ref{sequence} holds for $A=\frac{6{\size}}\lambda+1$.
\end{proof}

\section{Application to homogenization of the G-equation}\label{G_e}

In this section we prove a result about  the homogenization limit of solutions to the G-equation with random drift. 
The proof of this result is a corollary of Theorem~\ref{newmain} combined with standard arguments of the homogenization theory. We give these arguments here for convenience of the reader. 
We start with the notions needed to formulate our result.

We investigate the asymptotic behavior as $\ep \to 0$ of the solutions of the family of the initial value problems, parametrized by $\ep$.  Namely, we consider the  family of
Hamilton-Jacobi equations:
\begin{gather}
u^\ep_t + V_{\frac{t}{\ep}} \left(\frac{x}{\ep},\event \right) \cdot Du^\ep = \abs{Du^\ep} , 
 \quad t > 0 , \;\; x \in \R^{\size}, \label{Geqn} \\
u^\ep = u_0(x),\quad t = 0, \;\;x \in \R^{\size}, \no
\end{gather}
for the unknown $u^\ep=u^\ep(t,x,\omega)$,
where $u^\ep_t$ and $Du^\ep$ are the derivatives of $u^\ep$
with respect to $t$ and $x$, respectively.
%
%
Here $\event$ is an elementary event (realization) in the sample space: $\event \in \Events$.
We assume the sample space is a part of the probability triple $(\Events, \mathcal{F}, \Pm)$, 
where $\mathcal{F}$ is the $\sigma$-algebra of measurable events, and
$\Pm$ is the probability measure. The velocity
\[
 V_t:\R^{\size+1} \times \Events \to \R^{\size}  
\]
is a random field, a family of random variables parametrized  by $x$ and $t$.
All random variables are assumed Borel measurable.

If $V_t$ is locally Lipschits, then, by, e.g. Exercise 3.9 in~\cite{Bardi}, we are guaranteed that the viscosity solutions of the G-equation~~\eqref{Geqn}
  are unique in the space of bounded and uniformly continuous functions for every fixed $\event$.
 These solutions $u^\ep(t,x,\event)$ of~\eqref{Geqn} are random functions in $x$ and~$t$.
Our objective is to determine assumptions on $V_t(x,\event)$ that imply {\it the law of large numbers}: 
 $u^\ep(t,x,\event) \to \bar u(t,x)$  with probability one as $\ep \to 0$, and characterize the deterministic limit $\bar u(t,x)$ as a solution of another 
 {\it homogenized} initial value problem.  In order to determine this homogenized initial value problem, we will find a deterministic time-independent function $\bar H: \R^\size \to \R^+$ such that it 
 is positively homogeneous of degree one, that is $\bar H(\lambda p) = \lambda \bar H(p)$ for all $\lambda > 0$ and $p \in \R^{\size}$, and verify that 
$\bar u$ is the unique viscosity solution of the initial value problem
\begin{gather}\label{limit_h}
\bar u_t = \bar H(D \bar u), \quad x \in \R^{\size},\; t > 0,  \\
\bar u(0,x) = u_0(x) , \quad x \in \R^{\size}. \no
\end{gather}
The solutions  of~\eqref{Geqn} have a control-representation formula~\eqref{controlrep}. Similarly solutions of~\eqref{limit_h} are given by the Hopf-Lax formula~\cite{H,L}
\be\label{HopfLax}
\bar u(t,x) = \max \{ u_0(y) : \bar{T} \left( x-y \right) \leq t \},
\ee
where
\[
\bar{T}(v) = \sup \left \{ v \cdot q \;:\;\; q \in \R^\size, \bar{H}(q)=1 \right  \}.
\]
The  following two definitions are needed to state our assumptions on $V_t(x,\event)$.

\begin{defin}  
We say that $V_t(x, \event)$ is space-time stationary  if there is 
an action of $\R^{n+1}$ on $\Omega$,
denoted by $y\mapsto \pi_y: \Events \to \Events$, $y=(x, t) \in \R^{\size+1}$,
such that the action is measure-preserving:
\be\label{stat}
\Pm(\pi_y(A)) = \Pm(A), \quad \forall A \in \mathcal{F}, \ y\in \R^{\size+1},
\ee
and
\be\label{stat2}
V_{t_0} (x_0, \pi_y \event) = V_{t_0+t} (x_0+x, \event),  \quad
\forall x_0\in\R^n, \ t_0\in\R, \ y=(x,t)  \in \R^{\size+1}.
\ee
\end{defin}

\begin{defin}  
Define 
\be
\mathcal G_{t^+} := \sigma \left\{V_{s}(x, \event) : s \geq  t, x\in  \R^\size \right\},
\quad
\mathcal G_{t^-}:= \sigma \left\{V_{s}(x, \event) : s \leq  t, x\in   \R^\size \right\},
\ee
where $\sigma\{\dots\}$ denotes the $\sigma$-algebra on $\Omega$ generated
by the given family of random variables.
We say $V_t$ has {\em finite range of  time dependence} if
\be\label{finite_range}
\exists \Range>0 \hbox{ such that }\mathcal G_{t^+}  \hbox{ and } \mathcal G_{s^-} 
\hbox{ are independent when }  t-s   \geq \Range.
\ee
\end{defin}

We state the result in two essentially equivalent ways.

\begin{theorem}\label{limit_shape}
Suppose that a random vector field
$V_t\colon\R^{\size+1} \times \Events \to \R^{\size}$ is time-space stationary~\eqref{stat}--\eqref{stat2}, 
has finite range of time dependence~\eqref{finite_range},
$V_t(\cdot,\event)$ is locally Lipschitz and incompressible for all $t$ and $\event$,
and has the following uniform bounds:
\be\label{uniform_M}
 M:=1+\sup_{t,x,\event}|V_t(x,\event)|<\infty ,
\ee
\be\label{uniform_Delta}
 \Delta:=\inf_{L>0} \sup_{t,x,\event}
 \left\| \frac{1}{L^{\size}} \int_{[0,L]^{\size}} V_t(x+y,\event) \, dy \right\| <1 .
\ee
Then there exists a convex body $W\subset\R^n$
such that $B_{1-\Delta}(0)\subset W\subset B_M(0)$
and
$$
  \lim_{t\to\infty} d_H(t^{-1} \RR_t(\event) , W) = 0
$$
for a.e.\ $\event\in\Events$,
where $\RR_t(\event)$ is the reachable set 
from $(0,0)$ at time $t$ (see Section \ref{prelim})
of the flow $V_t(x,\event)$ 
and
$d_H$ denotes the Hausdorff distance.
\end{theorem}

\begin{theorem}\label{Gthm}
Let $V_t:\R^{\size+1} \times \Events \to \R^{\size}$ be a random vector field satisfying
the same assumptions as in Theorem \ref{limit_shape}.
Then there exists a positively one-homogeneous convex
Hamiltonian function $\bar H:\R^{\size} \to [0,\infty)$ 
with $1-\Delta \leq \bar H(p)/\abs{p} \leq M$
such that the following holds with probability one:
For every bounded uniformly continuous function $u_0:\R^n\to\R$ one has
\be
\forall T>0~\forall R>0~\lim_{\ep \to 0} \sup_{t \in [0,T]} \sup_{\abs{x} \leq R}  \left | u^\ep(t,x,\event) - \bar u(t,x) \right| = 0, \label{homconv}
\ee
where $u^\ep$ and $\bar u$ are the unique viscosity solutions of \eqref{Geqn} and \eqref{limit_h}, respectively.
\end{theorem}

\begin{rmk}
Theorems \ref{limit_shape} and~\ref{Gthm} are also true if we 
request $V_t$ to be merely integer stationary. 
This means that~\eqref{stat}-\eqref{stat2} holds for $y =(x, t) \in \mathbb{Z}^{\size+1}$ only.
Here is an example
of an integer stationary and finite range dependent flow $V_t(x,\event)$ that satisfies the conditions of Theorem~\ref{limit_shape}. 
Take any two deterministic incompressible
vector fields $V^1_t(x)$ and $V^2_t(x)$ with compact support in $\R^{\size+1}$. 
The incompressibility and compact support imply that 
\begin{equation}\label{dos_dos}
 \int_{\R^{\size}} V_t^i(x) \, dx =0, \qquad i=1,2, 
\end{equation}
for every $t$.
Consider a family of Bernoulli trials, that is 
$\zeta_{j k}(\event)$, $j \in \Z^{\size}$, $k \in \Z$ are independent identically distributed random variables such that $\zeta_{j k}= 1$ or   $\zeta_{j k}= 0$ with probability $1/2$.  Set
\[
V_t(x,\event) 
= \sum_{j \in  \Z^{\size}, k \in \Z}\left( \zeta_{jk}(\event) V^1_{t+k}(x+j) + (1- \zeta_{jk}(\event)) V^2_{t+k}(x+j) \right).
\]
The identity \eqref{dos_dos} implies that this random field satisfies \eqref{uniform_Delta} with $\Delta=0$.
\end{rmk}

\begin{rmk} Using Theorem~\ref{newmain} and Corollary \ref{new_cor} we can prove 
the conclusions of Theorems \ref{limit_shape} and~\ref{Gthm} if, instead of  finite range dependence and stationarity,
we impose other assumptions on  $V_t$. We are aware of two approaches. 
\begin{itemize}
\item If $V_t$ is periodic in $x$ and  random, statistically stationary and ergodic with 
respect to $t$, then the homogenization limit can be proven by an argument given in~\cite{JST}. 

\item If  $V_t$ is periodic in $t$ and  
random, statistically stationary and ergodic with 
respect to $x$, then the homogenization limit can be proven by an argument given in~\cite{NN}. 
\end{itemize}
Note that the level-set equation (\ref{Geqn}) is used 
as a model for turbulent combustion in the regime of thin flames \cite{Wlms85, Pet}.
In this model, the level sets of $u^\ep$ represent the flame surface, and 
$V_t$ is the velocity of the underlying fluid (assumed to be independent of $u^\ep$).
Spatial or temporal periodicity is rarely 
observed in unsteady turbulent flows. Thus, in the context of unsteady turbulent flows 
it is more relevant to assume the velocities are time-space stationary and have finite range of time dependence. 
\end{rmk}

We prove Theorems \ref{limit_shape} and~\ref{Gthm}  for a time-space stationary random vector field. 
Generalization to the  integer stationary case is straightforward.
We denote by $\RR_t(x_0, t_0, \event)$ the reachable set from $(x_0,t_0)$ at time $t_0+t$
of the flow $V_t(x,\event)$.
Note that $\RR_t(\event)=\RR_t(0, 0, \event)$.

Observe that
\be\label{Mt-ball}
  \RR_{t}(x_0,t_0,\event)  \subset B_{M t}(x_0) 
 \qquad
 \forall t>0, \  x_0\in\R^\size, \ t_0\in\R, \ \event\in\Events.
\ee
Define $\Lambda=\frac 2{1-\Delta}$.
Corollary \ref{new_cor} implies that there is a positive integer  $\tau_0 \in \mathbb{N}$  such that
\be\label{t/2-ball}
 B_{t/\Lambda}(x_0) \subset \RR_{t}(x_0,t_0,\event) 
 \qquad
 \forall t\ge \tau_0-1, \  x_0\in\R^\size, \ t_0\in\R, \ \event\in\Events .
\ee
Here we use \eqref{uniform_M}, \eqref{uniform_Delta} and Remark \ref{newcor_uniform}
to ensure that $\tau_0$ is independent of~$\event$.
We assume that $\tau_{0}>\Range$ where $\Range$ is the range of time dependence from \eqref{finite_range}.

The relation \eqref{t/2-ball} implies that $x_0\in \RR_{t}(x_0,t_0,\event)$ for all $t\ge\tau_0-1$.
Therefore
\be\label{Rt-monotone}
 \RR_{t_1}(x_0,t_0,\event) \subset \RR_{t_1+t}(x_0,t_0,\event)
 \qquad \forall t\ge\tau_0-1, \ t_1\ge 0, \  x_0\in\R^\size, \ t_0\in\R, \ \event\in\Events .
\ee

For $x_0,v \in \R^\size$, $t_0\in\R$ and $\event\in\Events$, define the travel-time
\be\label{def_tt}
\tau( x_0, t_0,v,  \event)= \inf \{ t \in \N : x_0 + v \in \RR_{t}(x_0 ,t_0,\event)  \} +\tau_0.
\ee
Set $\tau(v, \event) = \tau( 0, 0, v, \event)$. 
Note that for any $N\in\N$ the event $\{\event\in\Omega:\tau( x_0, t_0,v,  \event)=N\}$
is determined by the restriction of $V_t$ to the time interval $[t_0,t_0+N-\tau_0]$.

By \eqref{Mt-ball} and \eqref{t/2-ball},
the random variable $\tau(v, \event)$ grows linearly in $v$
and moreover
\be\label{m_est}
 \frac{|v|}{M}  \leq \tau(x_0, t_0, v, \event)   \leq \Lambda |v| + 2\tau_0
\ee
for all $x_0$, $t_0$, $v$, $\event$.
This estimate is the main ingredient of the first steps of the proof.
We also need a number of technical estimates.
By \eqref{Rt-monotone} we have
\be\label{reach_after_tau}
x_0 +v \in \RR_{t}(x_0, t_0)
\qquad \forall t \geq \tau(x_0, t_0, v, \event)-1
\ee
and
\be\label{no_reach_before_tau}
\tau(x_0, t_0, v, \event) \le t_1+2\tau_0 \qquad \text{if } x_0+v\in \RR_{t_1}(x_0, t_0) .
\ee
For any $x_0, x_1,v_0,v_1\in\R^n$ and $t_0\in\R$ we have
\be\label{quadrangle}
 \tau(x_0,t_0,v_0,\event) \le \tau(x_1, t_0+T , v_1 , \event) + 2T
\qquad \forall T\ge \Lambda|x_1-x_0|+\Lambda|v_1-v_0|+\tau_0 .
\ee
Indeed, $(x_1,t_0+T)$ is reachable from $(x_0,t_0)$ by \eqref{t/2-ball}.
Then the point $x_1+v_1$ is reachable from $(x_1,t_0+T)$
at time $t_1=t_0+T+\tau(x_1, t_0+T , v_1 , \event)-\tau_0$.
Then, by \eqref{t/2-ball}, $x_0+v_0$ is reachable from $(x_1+v_1,t_1)$ at any time
$t_2 \ge t_1+ T-1$.
Choosing $t_2$ such that $t_2-t_0$ is an integer and $t_2\le t_1+T$ yields \eqref{quadrangle}.

Our preliminary goal is to obtain the asymptotic shape of the reachable set. 
  This is analogous to  ``shape theorems" for the first-passage time in percolation theory, and we proceed with similar arguments. 
 
 \begin{lemma}\label{lower_}
 There exists a positively $1$-homogeneous convex function
 \[
 \overline{T}: \R^\size \to \R^+
 \]
satisfying
\be\label{estimate_}
 \frac{|v|}{M}  \leq \overline{T}(v)   \leq  \frac{|v|}{1-\Delta}
 \ee
for all $v\in\R^n$ and such that the following holds: 
\begin{itemize}
\item[i.] For any $v \in \R^\size$, $x_0 \in \R^\size$, $t_0 \in \R$, 
\[
\limsup_{\lambda \to \infty} \frac{1}{\lambda} \tau(\lambda x_0, \lambda t_0, \lambda v, \event) =\overline{T}(v) \hbox{ almost surely} .
\]

\item[ii.] For any $v \in \R^\size$, $x_0 \in \R^\size$, $t_0 \in \R$, 
\[
 \frac{1}{\lambda} \tau(\lambda x_0, \lambda t_0, \lambda v, \event)    \to \overline{T}(v)
\]
in probability as $\lambda \to \infty$, that is
\be\label{zlimp}
  \lim_{ \lambda \to\infty} \Pm \left\{ \event : \left|\frac{1}{\lambda}  \tau(\lambda x_0, \lambda t_0, \lambda v, \event)   - \overline{T}(v)\right| \ge \ep \right\} = 0
\ee
for every $\ep>0$.
\end{itemize}
\end{lemma}

\begin{proof}
Fix $x_0,v_1,v_2\in\R^n$, $t_0\in\R$, and define $\tau_1(\event)=\tau(x_0, t_0, v_1, \event)$.
By \eqref{reach_after_tau} and the definition of $\tau$ we have the following sub-additivity relation: 
\be\label{super2}
\tau(x_0, t_0, v_1+v_2, \event) \leq  \tau_1(\event) + \tau(x_0 +v_1, t_0+ \tau_1(\event), v_2, \event). 
\ee
The two terms in the right-hand side of \eqref{super2}
are independent random variables and
they have the same distributions as $\tau(v_1, \cdot)$ and $\tau(v_2, \cdot)$,
respectively. 
To show this, fix any $N_1, N_2 \in \N$
and consider events 
$$
A_{N_1}=\{ \event :  \tau_1(\event) =N_1 \},
$$
and
$$
B_{N_1, N_2}=\{\event : \tau(x_0 +v_1, t_0+ N_1, v_2, \event) =  N_2 \} .
$$
Due to the space-time stationarity, their probabilities are equal to those of
$\{ \tau(v_1, \cdot) =  N_1 \}$ and $\{ \tau(v_2, \cdot) =N_2 \}$, respectively.
The event $A_{N_1}$ is determined by $V_t(x ,\event)$ for $t \leq t_0+N_1-\tau_0$
and $B_{N_1, N_2}$  is determined by $V_t(x ,\event)$ for $t \geq t_0+N_1$.
Since $\tau_0 > \aleph$, the finite range of time dependence implies that $A_{N_2, N_1}$ and  $B_{N_2}$ are independent.
Thus
\begin{multline}\label{tau-indep}
\mathbb{P} (\{\event :  \tau_1(\event)=N_1  \text{ \bf and } \tau(x_0 +v_1, t_0+ \tau_1(\event), v_2, \event) =  N_2\} )
= \mathbb{P} (A_{N_1} \cap B_{N_1,N_2} ) 
\\
=\mathbb{P} (A_{N_1}) \,  \mathbb{P}(B_{N_1,N_2} ) 
= \mathbb{P} (\{ \tau(v_1, \cdot) =  N_1\}) \, \mathbb{P} (\{ \tau(v_2, \cdot) =N_2 \}) .
\end{multline}
Summing over either $N_2$ or $N_1$ we obtain that $\tau_1(\event)$ and $\tau(x_0 +v_1, t_0+ \tau_1(\event), v_2, \event)$
have the same distributions as $\tau(v_1, \cdot)$ and $\tau(v_2, \cdot)$, respectively;
furthermore \eqref{tau-indep} shows that they are independent.

Therefore, from  \eqref{super2} we have
\be\label{sub_add}
 \E(\tau(v_1+v_2,\cdot)) \le   \E(\tau(v_1,\cdot)) + \E(\tau(v_2,\cdot)).
\ee
This implies that there exists a limit
\be \label{eta}
\overline{T}(v) := \lim_{\lambda\to\infty} \frac{\E(\tau( \lambda v,\cdot))}{\lambda} 
 = \inf_{\lambda>0} \frac{\E(\tau( \lambda v,\cdot))}{\lambda}.
\ee
The function $\overline{T}$ is 1-homogeneous by definition.
By \eqref{sub_add},  $\overline{T}$ is sub-additive and hence convex.
The inequality \eqref{m_est} implies that $|v|/M\le \overline{T}(v)\le \Lambda|v|$.
Moreover, by Corollary \ref{new_cor} for every $a>\frac 1{1-\Delta}$
there is a constant $C>0$ such that $\tau(v,\event)\le a|v|+C$
for all $v\in\R^n$ and $\event\in\Events$.
Hence $\overline T(v)\le a|v|$ for all $a>\frac 1{1-\Delta}$
and \eqref{estimate_} follows.

Fix $v\in\R^n$ and arbitrary sequences $\{x_k\}\subset\R^n$ and $\{t_k\}\subset\R$, $k\in\N$.
For each $k$, 
define finite sequences $\xi_{k, m}$ and $t_{k,m}$, $1 \leq m \leq k$,
of random variables by induction as follows:
$$
 \xi_{k, m}(\event) = \tau(x_k +(m-1) v, t_{k, m}(\event), v, \event),
$$
where
$$
 t_{k, m}(\event) = t_k+\sum_{i=1}^{m-1} \xi_{k, i}(\event),
$$
in particular $t_{k,1}(\event)=t_k$.
Note that for any $N\in\N$ the event $\{\event:t_{k, m}(\event)=t_k+N\}$ is determined by the values $V_t(x,\omega)$
for $t\in[t_k, t_k+N-\tau_{0}]$ only.
As in the above discussion of the terms in \eqref{super2}, one sees that
for each fixed $k$ the random variables $\xi_{k,m}$, $1\le m \leq k$,
are independent  and  have the same distribution as $\tau(v,\cdot)$.
Since  $\xi_{k, m}$ are uniformly bounded (see \eqref{m_est}),
the strong law of the large numbers for triangular arrays applies to them, and we obtain that
\be\label{sll2}
\lim_{k \to \infty} \frac{1}{k} \sum_{m=1}^k \xi_{k, m}(\event) 
=  \E(\tau(v, \cdot)) 
, \hbox{ almost surely. }
\ee
As  in~\eqref{super2} we have sub-additivity
$$
  \tau(x_k, t_k, k v, \omega) \le  \sum_{m=1}^k \xi_{k, m}(\event)
$$
for all $k\in\N$ and $\event\in\Events$.
This and \eqref{sll2} imply that
\be\label{sll3}
 \limsup_{k \to \infty} \frac{\tau(x_k, t_k, k v, \event)}{k} \le  \E(\tau(v, \cdot))
, \hbox{ almost surely}.
\ee

Now we prove the two main assertions of the lemma.
Fix $x_0,v\in\R^n$, $t_0\in\R$, and $\ep>0$.
By \eqref{eta} there exists $\lambda_0 >0$ such that
$$
 \overline{T}(v) \le \frac{\E(\tau(\lambda_0 v ,\cdot))}{\lambda_0} \le (1+\ep)\overline{T}(v) .
$$
For $\lambda\ge\lambda_0$ let $k\in\N$ be such that
$k\lambda_0\le \lambda<(k+1)\lambda_0$. 
We apply~\eqref{quadrangle} 
to $\lambda x_0$, $k\lambda_0 x_0$, $\lambda v$, $k\lambda_0v$, $\lambda t_0$
in place of $x_0$, $x_1$, $v_0$, $v_1$, $t_0$, respectively,
with $T=T_0+ (k \lambda_0 -\lambda) t_0$ where
$T_0=\Lambda \lambda_0 |x_0| + \Lambda \lambda_0 |v|+ \lambda_0 |t_0|+\tau_0$.
This implies that
$$
 \tau(\lambda x_0, \lambda t_0, \lambda v, \event) 
 \le  \tau(k \lambda_0 x_0, k \lambda_0 t_0 +  T_0 , k \lambda_0 v,\event)  + 2 T_0 + 2 \lambda_0 |t_0|
$$
where the last term comes from the estimate $|k\lambda_0-\lambda|\le \lambda_0$.
Therefore
$$
 \limsup_{\lambda \to\infty} \frac{\tau( \lambda x_0, \lambda t_0, \lambda v,\event)}{\lambda} 
 \le  \limsup_{k\to\infty} \frac{\tau( k\lambda_0 x_0, k\lambda_0 t_0 + T_0, k\lambda_0 v,\event)}{k\lambda_0} .
$$
By \eqref{sll3} applied to $x_k=k\lambda_0 x_0$, $t_k=k\lambda_0 t_0+T_0$, and $\lambda_0v$ in place of $v$,
the right-hand side is bounded by
$\E(\tau(\lambda_0v,\cdot))/\lambda_0$ almost surely.
Thus
$$
 \limsup_{\lambda \to\infty} \frac{\tau( \lambda x_0, \lambda t_0, \lambda v,\event)}{\lambda} 
 \le \frac{\E(\tau(\lambda_0v,\cdot))}{\lambda_0}
  \le (1+\ep)\overline{T}(v), 
  \quad\text{almost surely}.
$$
Since $\ep$ is arbitrary, it follows that
\be\label{limsup_tau_above}
\limsup_{ \lambda \to\infty} \frac{\tau(\lambda x_0, \lambda t_0, \lambda v, \event)}{\lambda} \le \overline{T}(v),
\quad\text{almost surely}.
\ee
By the space-time stationarity and \eqref{eta},
\be\label{E_tau_below}
 \E\left(\frac{\tau(\lambda x_0, \lambda t_0, \lambda v,\cdot)}{\lambda}\right) 
 = \E\left(\frac{\tau(\lambda v,\cdot)}{\lambda}\right) 
 \ge \overline{T}(v) .
\ee
Since $\tau(\lambda x_0, \lambda t_0, \lambda v,\cdot)/\lambda$
is bounded above by $\Lambda|v|+\tau_0$ for all $\lambda\ge 1$,
\eqref{limsup_tau_above},~\eqref{E_tau_below} and Fatou's lemma imply that
$$
 \limsup_{\lambda \to\infty} \frac{\tau(\lambda x_0, \lambda t_0, \lambda v,\event)}{\lambda} 
 = \overline{T}(v), \quad\text{almost surely},
$$
and $\tau(\lambda x_0, \lambda t_0, \lambda v,\cdot)/\lambda$ converges to $\overline{T}(v)$ in probability. 
\end{proof}

\begin{defin}
Let $\overline T$ be the function constructed in Lemma \ref{lower_}.
Define the {\em effective reachable set}
\[
W_t =  \left \{v \in \R^\size :  \overline{T}(v) \leq t \right \}.
\]
\end{defin}

Note that $W_t=t\cdot W_1$ and $W_1$ is a convex body satisfying $B_{1-\Delta}(0)\subset W_1\subset B_M(0)$.
We are going to show that the reachable set $\RR_{t}(x_0,t_0,\omega)$ 
for large $t$ is close to the set $x_0+W_t$ in a certain sense.
We introduce the following quantity measuring the difference between these sets.

\begin{defin}
For $x_0\in\R^n$, $t_0\in\R$, $t\ge\tau_0$ and $\event\in\Events$ define
$$
\begin{aligned}
 \rho^+(x_0,t_0,t,\event) &= \inf\{ \ep>0 : \RR_{t}(x_0,t_0,\event)\subset x_0+(1+\ep)W_t \} , \\
 \rho^-(x_0,t_0,t,\event) &= \inf\{ \ep>0 : x_0+(1+\ep)^{-1}W_t \subset \RR_{t}(x_0,t_0,\event) \}
\end{aligned}
$$
and
$$
 \rho(x_0,t_0,t, \event) = \max\{  \rho^+(x_0,t_0,t,\event) , \rho^-(x_0,t_0,t,\event) \} .
$$
\end{defin}

Note that the statement of Theorem \ref{limit_shape} is equivalent
to the property that
$$
\lim_{t\to\infty}\rho(0,0,t,\event)=0, \quad\text{almost surely}.
$$



\begin{lemma}\label{limsup-prob} For any fixed $R>0$,
\be\label{est1}
\lim_{t\to\infty}  \sup_{|x_0|\le Rt} 
 \rho^-(x_0, 0, t,\event) = 0 \hbox{ almost surely}
\ee
and
\be\label{est2}
\lim_{t \to \infty} \sup_{|x_0| \leq Rt} 
\rho^+(x_0,0,t,\event) = 0 \hbox{ in probability,}
\ee
that is for any $\ep>0$,
\be\label{est2a}
 \Pm \left\{ \event : \forall x_0 \in B_{Rt}(0),  \RR_{t} (x_0, 0, \event) \subset  x_0 +(1+\ep) W_t \right\} \to 1
 \quad \text{ as } t \to\infty.
\ee
\end{lemma}

\begin{proof}
To prove \eqref{est1}, fix $R>0$ and $\ep>0$ and choose $\ep$-nets  $\{y_i\}_{i=1}^N$ in the ball $B_{R}(0)$
and $\{v_j\}_{j=1}^K$ in the effective 1-reachable set $W_1$.
For every $x_0\in B_{Rt}(0)$ and $v\in W_t$ there exist $i$ and $j$ such that
$|x_0-ty_i|<t\ep$ and $|v-tv_j|<t\ep$.
Assuming that $t\ge\ep^{-1}\tau_0\ge 2\Lambda^{-1}\ep^{-1}\tau_0$, we see from \eqref{quadrangle} that 
$$
 \tau(x_0,0,v,\event) \le \tau(ty_i,3\Lambda t\ep,tv_i,\event) + 6\Lambda t\ep
$$
for all $\event\in\Events$. Hence
\[
\sup_{|x_0| \leq Rt, v\in W_t} \tau(x_0, 0, v, \event) 
\leq \max_{i,j}  \tau( ty_i,  3\Lambda t\ep, tv_j, \event)+6\Lambda t\ep
\]
for all $t\ge\ep^{-1}\tau_0$ and $\event\in\Events$.
By Lemma~\ref{lower_}  (part 1)
\[
\limsup_{t\to\infty} \max_{i,j}  \frac{1}{t} \tau( ty_i, 3\Lambda t\ep , tv_j, \event) = \max_j\overline{T}(v_j)\le 1,
\quad\text{almost surely}.
\]
Thus
$$
\limsup_{t\to\infty} \sup_{|x_0| \leq Rt, v\in W_t} \frac1t \tau(x_0, 0, v, \event) 
 \le 1+6\Lambda\ep, \quad\text{almost surely}.
$$
By~\eqref{reach_after_tau} this implies that for every $\delta>0$ there is 
$s=s(\delta,\event)>0$ such that
$$
x_0 +  v  \in \RR_{t(1+ 6 \Lambda \ep +\delta)}(x_0, t_0)
$$
for all  $t \geq s$, $v \in  W_t$ and $|x_0| \leq Rt$.
Setting $\delta=\Lambda\ep$ we obtain that
$$
 \rho^-(x_0,0,t(1+ 7 \Lambda \ep),\event) \le 7 \Lambda \ep
$$
for all  $t \geq s=s(\Lambda\ep,\event)$ and $|x_0| \leq Rt$.
Therefore
$$
 \limsup_{t\to\infty} \sup_{|x_0|\le R't} \rho^-(x_0,0,t,\event) \le 7\Lambda\ep,
\quad\text{almost surely}
$$
where $R'=(1+ 7 \Lambda \ep)^{-1}R$.
Since $R$ and $\ep$ are arbitrary, \eqref{est1} follows.
To prove~\eqref{est2}, fix $R>0$ and $\ep>0$ and define
\[
 \Events_1(t)= \left\{ \event : \exists x_0 \in B_{Rt}(0), \rho^+(x_0,0,t,\event)>\ep \right\}.
\]
Let $\delta=\ep/32\Lambda$ and choose $\delta$-nets  $\{y_i\}_{i=1}^N$ in $B_{R}(0)$
and $\{v_j\}_{j=1}^K$ in $B_M(0)$.
Consider $\event\in\Events_1(t)$ where $t\ge\delta^{-1}\tau_0$.
By the definition of $\Events_1(t)$ there exist $x_0\in B_{Rt}(0)$ 
and $v\in\RR_t(x_0,0,\event)-x_0$
such that $v\notin(1+\ep)W_t$.
By \eqref{Mt-ball} we have $v\in B_{Mt}(0)$, hence
there exist $i$ and $j$ such that
$|x_0-ty_i|<\delta t$ and $|v-tv_j|<\delta t$.
These inequalities, \eqref{quadrangle}, and \eqref{no_reach_before_tau} imply that
\be\label{tau_ij}
 \tau(ty_i,-3\Lambda \delta t,tv_j,\event) 
 \le \tau(x_0,0,v,\event) + 6\Lambda \delta t\le t + 2\tau_0 + 6\Lambda \delta t\le (1+\ep/4) t .
\ee
Since $v\notin(1+\ep)W_t$, we have $\overline T(t^{-1}v) \ge 1+\ep$.
On the other hand,
$$
\overline T(t^{-1}v) \le \overline T(v_j) + \overline T(t^{-1}v-v_j)
\le \overline T(v_j) + \Lambda |t^{-1}v-v_j| 
\le \overline T(v_j) + \Lambda \delta 
\le \overline T(v_j) + \ep/4  
$$
by the sub-additivity of $\overline T$ and \eqref{estimate_}.
Therefore $\overline T(v_j)\ge 1+\ep/2$.
Hence, by \eqref{tau_ij},
$$
 \frac1t \tau(ty_i,-3\Lambda \delta t,tv_j,\event) \le 1 + \ep/4 \le \overline T(v_j) - \ep/4.
$$
Thus
$$
 \Pm(\Events_1(t)) \le \sum_{i,j} \Pm \left\{\event: \frac1t \tau(ty_i,-3\Lambda\delta t,tv_j,\event) \le \overline T(v_j) - \ep/4 \right\} 
$$
for all $t\ge \delta^{-1}\tau_0$.
By Lemma~\ref{lower_} (part 2), each summand in the right-hand side goes to 0 as $t\to\infty$.
Hence $\Pm(\Events_1(t))\to 0$ as $t\to\infty$ and \eqref{est2} follows.
\end{proof}

\begin{defin} Define the support function of $W_1$ (a.k.a. the effective Hamiltonian)
\[
\bar H(p) = \sup \left \{ p \cdot y \;|\;\; y \in W_1 \right \}.
\]
\end{defin}
Since $\bar H(p)$ is the supremum of a family of linear functions of $p$,
it is immediate that $\bar H$ is convex in $p$, and positively homogeneous of degree one.
Since $B_{1-\Delta}(0)\subset W_1 \subset B_M(t)$,
we have $(1-\Delta)|p|\le \bar H(p) \le M|p|$.
Similarly, we define the support functions of reachable sets.

\begin{defin}
For $p \in \R^{\size}$, $x_0\in\R^n$, $t_0\in\R$ and $\event\in\Events$ define
$$
 H_t(x_0,t_0,p,\event) = \sup \{  p \cdot (y-x_0)  \mid y\in\RR_{t}(x_0,t_0,\event) \} 
$$
and
$$
H_t(p, \event) =H_t(0, 0, p, \event).
$$
\end{defin}

Due to the space-time stationarity,
the random variable $H_t(x_0,t_0,p,\cdot)$ has the same distribution
as $H_t(p,\cdot)$.

\begin{lemma}\label{Lemma_3}
For any $p \in \Rm^{\size}$  and $R>0$, 
\be\label{Lemma_3_statement}
\limsup_{t \to \infty}  \sup_{|x_0|\le Rt} \frac{H_{t}( x_0, 0, p,\event) }{ t } \leq \bar{H}(p), \hbox{ almost surely.}
\ee
\end{lemma}

Here is an outline of the proof of Lemma~\ref{Lemma_3}.
First we adjust parameters in \eqref{Lemma_3_statement} to define
a more manageable random variable  $h(t, \event)$,
see \eqref{def_h4} and \eqref{def_h} below.
The advantages of $h(t, \event)$ over the original expression are
its sub-additivity and independence properties,
demonstrated in the course of the proof.
With the new variable $h(t, \event)$ the lemma is reduced to \eqref{R_R},
which we then prove in four steps.
In Step~1 we prove the sub-additivity \eqref{hsubadd}.
Unfortunately this sub-additivity is weaker than the classical one;
we only have a bound for $h(qt, \event)$ by a sum of $h_q(t, \event)$ where $h_q$ is another random variable parametrized by~$q\in\N$.
We overcome this difficulty by chaining random variables $h_q(t, \event)$ to $h(t, \event)$ in Step~2. Namely, we show in Step 2 that one can control
distributions of  $h_q(t, \event)$  by distribution of $h(t, \event)$,  see~\eqref{2qRt}.
Step~3 is the key one. There we prove almost sure convergence for $t$ ranging along a geometric progression, see \eqref{limsup-q^m}.
We do this by analysis of the probability distribution of $h(t, \event)$
using our stationarity and independence assumptions,  sub-additivity of  $h(t, \event)$ and its convergence in probability~\eqref{hprob}.
In our final Step 4 we show that the linear bound~\eqref{Hrelbound} on the growth of  $H_{t}( x_0, t_0, p,\event)$ is sufficient to deduce the convergence  for all $t \to \infty$.

\begin{proof}[Proof of Lemma~\ref{Lemma_3}] We begin with several preliminary observations.
By scaling it is sufficient to consider $p \in \R^\size$ with $|p|=1$.
We may also assume that $R \geq M$.
We fix such $p$ and $R$ for the rest of the proof.
Since $\RR_{t}(x_0,t_0,\event) \subset B_{Mt}(x_0)$, we have
\be \label{Hbound}
 H_t(x_0,t_0,p,\event) \le Mt . 
\ee
Moreover,
\be\label{Hrelbound}
  H_{t_1+t_2}(x_0,t_0,p,\event) \le H_{t_1}(x_0,t_0,p,\event) + Mt_2 
\ee
for all $t_1,t_2\ge 0$, since $\RR_{t_1+t_2}(x_0,t_0)$ is contained
in the $(Mt_2)$-neighborhood of $\RR_{t_1}(x_0,t_0)$.

For $x_0\in\R^n$, $t_0\in\R$, $t\ge\tau_0$, $\event\in\Events$, define
\be\label{def_h4}
 h(x_0,t_0,t,\event) = \sup_{|x-x_0|\le Rt}  H_{t-\tau_0}(x,t_0,p,\event)  + M\tau_0
\ee
and, for brevity,
\be\label{def_h}
 h(t,\event) = h(0,0,t,\event) .
\ee
For every $x_0 \in B_{Rt}(0)$ we have
$H_t(x_0, 0, p,\event) \le h(t,\event)$
by \eqref{Hrelbound} applied to $t_1=t-\tau_0$ and $t_2=\tau_0$.
Thus, in order to prove the Lemma it suffices to show that
\be \label{R_R}
\limsup_{t\to\infty} \frac{h(t,\event)}{t} 
 \le \bar H(p).
\ee

Let us now reformulate the convergence in probability from Lemma \ref{limsup-prob} in terms of $h(t, \event)$. We claim that, for every $\ep>0$,
\be\label{hprob}
 \Pm \left\{ \event : \frac{h(t,\event)}{t} > \bar H(p)+\ep \right\} \to 0 \quad\text{as $t\to\infty$}.
\ee
Indeed, by~\eqref{est2a} in Lemma \ref{limsup-prob} we have 
\be\label{hprob-2}
 \Pm \left\{ \event :  \forall x_0 \in B_{R t}(0), \  \RR_{t} (x_0, 0, \event) - x_0 \subset  (1+\ep) W_t
  \right\} \to 1 \quad\text{as $t\to\infty$} .
\ee
For every $\event$ satisfying the relation $\RR_{t} (x_0, 0, \event) - x_0 \subset  (1+\ep) W_t$ in~\eqref{hprob-2}, we have
\[
H_t(x_0, 0, p, \event) \le
\sup \left \{ p \cdot y \mid y \in (1+\ep) W_t \right \} =(1 +\ep) t \bar H(p) .
\]
Therefore, we can conclude from~\eqref{hprob-2} that
\[
 \Pm \left\{ \event : \forall x_0 \in B_{R t}(0), \ \frac{  H_t(x_0, 0, p, \event) }{t} \le (1+\ep)\bar H(p) \right\} \to 1 \quad\text{as $t\to\infty$},
\]
for every $\ep>0$, and \eqref{hprob} follows.

In order to state sub-additivity properties of $h(t,\event)$ we need one more definition. Fix $q\in\N$ and define
$$
 h_q(x_0,t_0,t,\event) = \sup_{|x-x_0|\le qRt}  H_{t-\tau_0}(x,t_0,p,\event)  + M\tau_0
$$
and
$$
 h_q(t,\event) = h_q(0,0,t,\event)
$$
for  $x_0\in\R^n$, $t_0\in\R$, $t\ge\tau_0$, $\event\in\Events$.
(The reader may notice that for $q=1$ one has $h_q(t,\event)=h(t,\event)$
but we do not need this fact in the proof).
Observe that
\be \label{hbound}
  h_q(x_0,t_0,t,\event) \le Mt 
\ee
by \eqref{Hbound}.
We are now ready for our four steps.

\medskip {\it Step 1}.  {\it Sub-additivity of $h(t, \event)$.}
We show here that for every $q\in\N$, $t\ge\tau_0$, $\event\in\Events$,
\be\label{hsubadd}
  h(qt,\event) \le \sum_{k=0}^{q-1} h_{q}(0, kt, t, \event).
\ee

Indeed, let $\gamma\co [0, qt-\tau_0]\to\R^n$ be an admissible path for $V_t(x, \event)$
with $\gamma(0)\in B_{Rt}(0)$. 
To prove \eqref{hsubadd},
it suffices to verify that
\be\label{hsubadd1}
 (\gamma(qt-\tau_0)-\gamma(0))\cdot p \le \sum_{k=0}^{q-1} h_{q}(0,kt,t,\event) - M\tau_0
\ee
for every such path $\gamma$.
Observe that $\gamma(kt)\in B_{qRt}(0)$ for $k=0,\dots, q-1$ since 
$\gamma(0)\in B_{Rt}(0)$ and $|\dot\gamma|\le M\le R$.
Hence
$$
 (\gamma((k+1)t-\tau_0)-\gamma(kt))\cdot p
 \le H_{t-\tau_0}(\gamma(kt), kt, p,\event)
 \le h_{q}(0,kt, t, \event) - M\tau_0
$$
for each $k=0,1,\dots,q-1$.
We also have
$$
 (\gamma(kt)-\gamma(kt-\tau_0))\cdot p
 \le |\gamma(kt)-\gamma(kt-\tau_0)|
 \le M\tau_0
$$
for each $k=1,\dots,q-1$ .
Summing up these $2q-1$ inequalities yields \eqref{hsubadd1},
which implies \eqref{hsubadd}.

\medskip {\it Step 2}.  {\it Chaining of $h_q(t, \event)$.}
The goal of this step is to show that there exists $N=N(q,n)\in\N$
such that
\be \label{2qRt}
 \Pm\{\event: h_{q}(t, \event) > \alpha \}
 \le N\cdot\Pm\{\event: h(t,\event) > \alpha \}
\ee
for all $\alpha\in\R$, $t\ge\tau_0$, $\event\in\Events$.

To prove this, observe that a ball of radius $qRt$ can be covered
by $N$ balls of radius $Rt$:
$$
 B_{qRt}(0) \subset \bigcup_{i=1}^N B_{Rt}(z_i)
$$
for some $z_1,\dots,z_N$, where $N$ is determined by $q$ and~$n$.
Therefore
$$
 h_{q}(t,\event) \le \max_{1\le i\le N} h(z_i, 0, t, \event) ,
$$
hence
$$
\Pm\{\event: h_{q}(t,\event) > \alpha \}
\le \sum_{i=1}^N \Pm\{\event: h(z_i, 0,t,\event) > \alpha \} .
$$
Due to the space-time stationarity, each summand in the last sum
equals $\Pm\{\event: h(t,\event) > \alpha \}$
and the inequality \eqref{2qRt} follows.

\medskip {\it Step 3}.  {\it Convergence along a geometric progression.} 
As we have mentioned earlier,  this is the key step. Recall that our goal is to prove \eqref{R_R}.
Here we prove that the same inequality with a small error term
holds for $t$ ranging along a geometric progression with
common ratio~$q$, see \eqref{limsup-q^m} below.

Fix $\ep>0$ and $q\in\N$, and let $N=N(q,n)$ from Step~2.
Define
$$
  f(t,\event) = \frac{h(t,\event)}{t} - \bar H(p) - \ep
$$
and
$$
  f_k(t,\event) =  \frac{h_{q}(0, kt, t, \event)}{t} - \bar H(p) - \ep .
$$
for all $t\ge\tau_0$, $\event\in\Events$, $k\in\{1,\dots,q\}$.
Note that $f_k(t,\event)\le M$ by \eqref{hbound}.

With this notation, \eqref{hsubadd} takes the form
\be\label{fsubadd}
 f(qt,\event) \le \frac{1}{q} \sum_{k=0}^{q-1} f_k(t,\event) .
\ee
The inequality \eqref{2qRt} along with the space-time stationarity imply that
\be \label{fkprob}
  \Pm\{ f_k(t,\event) > \alpha \} \le N\cdot \Pm\{ f(t,\event) > \alpha \}
\ee
for all $\alpha\in\R$.

Fix a positive $\delta<\frac1{2q^2N^2}$.
By \eqref{hprob},
$$
 \Pm \{\event : f(t,\event) > 0 \} \to 0  \quad\text{as $t\to\infty$} .
$$
Hence there exists $t_0\ge\tau_0$ such that
\be \label{fprob}
 \Pm \{\event : f(t,\event) > 0 \} < \delta \qquad \forall t\ge t_0.
\ee

Define
\be \label{p(t)}
  \Delta(t) = \Pm \left\{\event : f(t,\event) > \frac{M}{q} \right\} 
\ee
for all $t\ge\tau_0$.
We are going to estimate $\Delta(qt)$ in terms of $\Delta(t)$ using the above inequalities.

Assume that $t\ge t_0$ where $t_0$ is the same as in \eqref{fprob}.
The bound $f_k(t,\event)\le M$ and \eqref{fsubadd} imply the following property:
For every $\event\in\Events$ such that $f(qt,\event)>\frac{M}{q}$,
at least two of the terms $f_k(t,q)$ must be positive
and at least one of them must be greater than $\frac{M}{q}$.
Therefore
\be\label{qtprob}
 \Delta(qt) \le \sum_{i\ne j} \Pm \left\{\event: f_i(t,\event)>\frac{M}{q} \text{ \bf and } f_j(t,\event)>0 \right\} .
\ee
Observe that the random variables $f_i(t,\cdot)$ and $f_j(t,\cdot)$ are independent if $i\ne j$.
This follows from the finite range time dependence and
the fact that $f_k(t,\event)$ is determined by the restriction of the flow
to the time interval $[kt,(k+1)t-\tau_0]$.
Hence \eqref{qtprob} can be rewritten as
$$
 \Delta(qt) \le \sum_{i\ne j} \Pm \left\{\event: f_i(t,\event)> M/q \right\} \cdot \Pm \left\{ f_j(t,\event)>0 \right\} .
$$
This and \eqref{fkprob}, \eqref{fprob}, \eqref{p(t)} imply that
$$
 \Delta(qt) \le \sum_{i\ne j} N \Delta(t) \cdot N\delta = q(q-1)N^2 \delta \Delta(t) \le \frac{\Delta(t)}{2}
$$
where the last inequality follows from the choice of $\delta$.

By induction it follows that $\Delta(q^m t)\le 2^{-m}$
for all $t\ge t_0$ and $m\in\N$.
By the Borel-Cantelli Lemma and \eqref{p(t)}, this implies that
for every $t>0$
$$
 \limsup_{m\to\infty} f(q^mt,\event) \le \frac{M}{q}
$$
for a.e.\ $\event\in\Events$.
Substituting the definition of $f$
yields that
\be \label{limsup-q^m}
 \limsup_{m\to\infty} \frac{h(q^mt,\event)}{q^mt}
 \le \bar H(p) + \ep + \frac{M}{q}
\ee
for a.e.\ $\event\in\Events$.

\medskip 
{\it Step 4}.  {\it Convergence for all $t$.}
To finish the proof, choose a partition $1=t_1\le t_2\le\dots\le t_l=q$ of $[1,q]$
such that $t_{i+1}<(1+\ep)t_i$ for all $i<l$.
For every $t\ge q$ there exist positive integers $m\in\N$ and $i<l$ such that
$$
 q^m t_i \le t < q^m t_{i+1} < q^m t_i + \ep t .
$$
These inequalities and \eqref{Hrelbound} imply that
$$
 h(t,\event) \le h(q^m t_i,\event) + M\ep t, 
$$
and hence
$$
 \limsup_{t\to\infty} \frac{h(t,\event)}{t} 
 \le \limsup_{m\to\infty} \, \max_{1\le i<l} \frac{h(q^m t_i,\event)}{q^mt_i} + M\ep
 = \max_{1\le i<l} \, \limsup_{m\to\infty} \frac{h(q^m t_i,\event)}{q^mt_i} + M\ep .
$$
for all $\event\in\Events$.
This and \eqref{limsup-q^m} imply that
$$
 \limsup_{t\to\infty} \frac{h(t,\event)}{t} 
 \le \bar H(p) + \frac{M}{q} + (M+1)\ep
$$
for a.e.\ $\event\in\Events$.
Since this holds for all $\ep>0$ and $q\in\N$, the estimate~\eqref{R_R} follows. 
This finishes the proof of the lemma.
\end{proof}

\begin{lemma}\label{limsup-as_} For any fixed $R>0$
\be\label{est3}
\lim_{t\to \infty}  \sup_{|x_0| \leq Rt} \rho^+(x_0,0,t,\event) = 0 \quad\text{almost surely}.
\ee
 \end{lemma}
 
\begin{proof}
Fix $R>0$ and $\ep\in(0,1)$.
Since $W_1$ is a compact convex set, we have
$$
  W_1 = \{ x\in\R^n : x\cdot p \le \bar H(p),  \ \forall p\in\R^n \} .
$$
Furthermore there is a finite collection of vectors
$p_1,\dots,p_N\in\R^n$ with $|p_i|=1$ such that
$$
  \widetilde W_1 := \{ x\in\R^n :  x\cdot p_i \le \bar H(p_i) ,\ \forall i\} \subset (1+\ep) W_1 .
$$
By Lemma \ref{Lemma_3}, for almost every $\event\in\Events$
there exists $t_\event>0$ such that for all  $t>t_\event$ and $x_0\in B_{Rt}(0)$,
$$
(x-x_0)\cdot p_i \le (1+\ep) t \bar H(p_i), \quad \forall x\in\RR_t(x_0,0,\event), \forall i.
$$
This implies that
$$
 \RR_t(x_0,0,\event) - x_0 \subset (1+\ep) t \widetilde W_1 \subset (1+\ep)^2  W_t
$$
and therefore $\rho^+(x_0,0,t,\event)<(1+\ep)^2-1<3\ep$. 
Since $\ep$ is arbitrary, \eqref{est3} follows.
\end{proof}

\begin{proof}[\bf Proof of Theorems \ref{limit_shape} and~\ref{Gthm}]
Theorem \ref{limit_shape} follows by setting $W=W_1$
and applying \eqref{est1} and~\eqref{est3}.

To prove Theorem \ref{Gthm} we recall the control representation
\eqref{controlrep} for the solution of the G-equations.
For $x\in\R^n$, $t>0$ and $\event\in\Events$ define
$$
 \RR^-_t(x,\event) = \{ y\in\R^n : x\in\RR_t(y,0,\event) \} .
$$
The control representation for the solution of 
\eqref{Geqn} and \eqref{limit_h} have the form
$$
 u^\ep(t,x, \event) = \sup \{ u_0(y) : y \in \ep \RR^-_{t/\ep}(x/\ep, \event) \}
$$
and
$$
 \bar u(t,x) = \sup \{ u_0(y) : y \in x-W_t \} .
$$
Let $\delta > 0$, $h > 0$, and $R > 0$. From~\eqref{est1} and~\eqref{est3} 
we see that for almost every $\event\in\Events$
there exists $\ep_0 = \ep_0(\delta,R,h,\event) > 0$ so that for all $\abs{x} \leq R$, $t \geq h$, and $\ep \leq \ep_0$ we have
$$
\{ x - W_{t(1 -  \delta)} \} \subset \ep \RR^-_{t/\ep}(x/\ep, \event) \subset \{ x - W_{t(1 + \delta)}\},
$$
Therefore
\be \label{liminfeps}
\bar u(t(1 - \delta),x) \leq u^\ep(t,x, \event) \leq \bar u(t(1 + \delta),x). 
\ee
Since $\delta>0$ is arbitrary and $\bar u(t,x)$ is 
uniformly continuous, 
\eqref{liminfeps} implies that $u^\ep \to \bar u$ uniformly on compact sets in $(0,\infty) \times \R^\size$. 
To obtain the locally uniform convergence down to time $t = 0$,  we need the uniform $L^\infty$ bound on $V_t$ and uniform continuity of $u_0(x)$.
 Observe that 
$$
\sup_{t \in [0,h]} \abs{u^\ep(t,x, \event) - \bar u(t,x)}  \leq  \sup_{t \in [0,h]} \abs{u^\ep(t,x, \event) - u_0(x)}  + \sup_{t \in [0,h]} \abs{\bar u(t,x) - u_0(x)} .
$$
For any $y \in  \ep \RR^-_{t/\ep}(x/\ep, 0, \event)$ we have  $\abs{y - x} \leq M t $. Thus the first term on the right is bounded by
\be
\sup_{t \in [0,h]} \abs{u^\ep(t,x, \event) - u_0(x)} \leq \sup_{\substack{y \in \R^\size\\ \abs{y - x} \leq M h }} \abs{u_0(y) - u_0(x)} \leq \phi(M h ),
\ee
where $\phi$ is the modulus of continuity of $u_0(x)$.
 This and a similar bound on $\abs{\bar u(t,x) - u_0(x)}$ implies that
\be
\lim_{h \to 0} \left [ \limsup_{\ep \to 0} \sup_{\substack{x \in \R^\size\\ t \in [0,h]}} \abs{u^\ep(t,x,\event) - \bar u(t,x)}\right] = 0. \label{tlessh}
\ee

Combining~\eqref{liminfeps}~and~\eqref{tlessh}, we conclude that (\ref{homconv}) holds with probability one.
\end{proof}

\appendix

\section{Functions of bounded variation}\label{BV}

We collect here needed facts about functions of bounded variation (BV functions) in $\R^{\size}$, $\size \geq 2$.
We followed \cite{A} and~\cite{M}.

\begin{defin}[Proposition 3.6 and Definition 3.4 in~\cite{A}]
\label{d:variation}
Let $\Omega\subset\R^n$ be an open set
and $u \in L^1(\Omega)$.
The {\it variation} of $u$ in $\Omega$, denoted by $\Var(u,\Omega)$, is
\[
\Var(u,\Omega) = \sup \left\{  \int_\Omega u \hbox{ div} \phi  : 
\phi \in [C^1_c (\Omega)]^{\size}, \| \phi\|_{L^\infty} \leq 1 \right\}.
\]
Here and below $[C^1_c (\Omega)]^{\size}$ denotes 
the set of all compactly supported $C^1$ functions from $\Omega$ to $\R^{\size}$.

The space $BV(\Omega)$ consists of all functions  $u \in L^1(\Omega)$ with $\Var(u,\Omega)<\infty$.
It is equipped with the norm
\[
 \| u\|_{BV} = \int_{\Omega} |u| dx + \hbox{Var}(u,\Omega).
\]
\end{defin}

\begin{remark}
The distributional derivative $Du$ of a BV-function $u$
is a (vector-valued) finite Radon measure,
and $\Var(u,\Omega)=|Du|(\Omega)$.
We occasionally 
 write
$$
\Var(u,\Omega)= \int_{\Omega} |\nabla u|,
$$
where the right-hand side is understood in the sense of distributions.
 \end{remark}


\begin{defin}[Definition 3.35 in~\cite{A}]\label{perim}
The {\em perimeter} $P(E,\Omega)$ of a measurable set $E \subset \R^n$
in an open set $\Omega\subset\R^n$ is defined by
\[
P(E,\Omega) =  \hbox{Var}(\chi_E, \Omega) 
= \sup \left\{  \int_E   \hbox{ div}  \phi : \phi \in [C^1_c (\Omega)]^{\size}, \| \phi \|_{L^\infty} \leq 1   \right\}.
\]
We denote $P(E)=P(E,\R^n)$.
\end{defin}

In all cases of interest in this paper the set $E$ is bounded.

\begin{defin}[Reduced boundary, Definition 3.54 in~\cite{A}]\label{reduced}
Let $E \subset \mathbb{R}^{\size}$ be a set of finite perimeter. 
The {\it reduced boundary} $\partial^*E$ of $E$ is the collection of points 
$x \in $supp$(|D \chi_E|)$ such that the limit
\begin{equation}\label{oh}
\nu_E(x) =\lim_{\rho \to 0} \frac{\int_{B_\rho(x)} \nabla \chi_E}{\int_{B_\rho(x)} |\nabla \chi_E|}
\end{equation}
exists in $\mathbb{R}^{\size}$ and satisfies $|\nu_E(x)|=1$.
The integrals here are understood in the sense of distributions.
The function $\nu_E: \partial^* E \to S^{\size-1}$ is called the generalized inner normal to $E$. 
\end{defin}

\begin{thm}[De Giorgi Theorem, Theorem 15.9 in~\cite{M}]\label{dg}
If $E \in \mathbb{R}^{\size}$ is a set of finite perimeter, then the reduced boundary
$\partial^* E$ is $\HA^{\size-1}$-rectifiable and
$$
  P(E,\Omega) = \HA^{\size-1}(\Omega \cap \partial^* E)
$$
for every open set $\Omega\subset\R^n$.
\end{thm}

 

Recall that $I_r=[-r,r]^n$ is a cube with edge length $2r$ and $I_r^\circ$
denotes its interior.

\begin{thm}[Relative isopertimetric inequality in the cube]\label{iso} 
If $E$ is a set of finite perimeter in~$\R^{\size}$, 
then for every $r>0$
\begin{equation}\label{isso}
\min\left(| E \cap I_r |, | I_r \setminus E| \right)^{\frac{\size-1}{\size}} 
\leq C P( E,  I_r^\circ) =  C \HA^{n-1}(\partial^* E \cap I_r^\circ), 
\end{equation}
where $C$ is a constant depending on $n$ only.
\end{thm}

\begin{proof}
This inequality is standard
but we could not find exactly this formulation in the literature.
For the sake of completeness we include a proof here.

Every $u\in BV(I_r^\circ)$ satisfies the following Sobolev inequality (see e.g.\ Remark 3.50 in \cite{M}):
there is a constant $C_1=C_1(n)$ such that
\be\label{sob}
\left(\int_{I_r} |u-\overline u|^{\frac{\size}{\size-1}} \right)^{\frac{\size-1}{\size}} \leq C_1 \hbox{Var}(u, I_r^\circ),
\ee
where $\overline u$ denotes the average of $u$ over $I_r$:
$$
  \overline u = \frac1{|I_r|} \int_{I_r} u .
$$
The fact that $C_1$ does not depend on $r$ follows from a scaling argument.

Let $u=\chi_E$, then 
$\overline u = \frac{|E \cap I_r|}{|I_r|}$
and $1-\overline u = \frac{|I_r\setminus E|}{|I_r|}$,
hence
\[
\int_{I_t} |u - \overline u|^{\frac{\size}{\size-1}} \, dx 
= \left(\frac{|I_r\setminus E|}{|I_r|}  \right)^{\frac{\size}{\size-1}} |E \cap I_r| + 
\left(\frac{|E \cap I_r|}{|I_r|}  \right)^{\frac{\size}{\size-1}} |I_r \setminus E|. 
\]
Therefore
\[
\left( \int_{I_r} |u - \overline u|^{\frac{\size}{\size-1}} dx \right)^{\frac{\size-1}{\size}} 
\geq  \frac{1}{|I_r|} \left( |E \cap I_r|^{\frac{\size}{\size-1}}  +  |I_r \setminus E|^{\frac{\size}{\size-1}}  
\right)^{\frac{\size-1}{\size}}  \min\left(| E \cap I_r |, | I_r \setminus E| \right)^{\frac{\size-1}{\size}} 
\]
\[
\geq  \frac12 \min\left(| E \cap I_r |, | I_r \setminus E| \right)^{\frac{\size-1}{\size}} .
\]
This and the Sobolev inequality~\eqref{sob} implies the inequality in~\eqref{isso}. 
The equality in~\eqref{isso} holds due to the De Giorgi Theorem~\ref{dg}.
\end{proof}

\begin{cor}\label{isopercube-lite}
If $E$ is a set of finite perimeter in $\R^n$, then for every $r>0$
$$
 \min\left(| E \cap I_r |, | I_r \setminus E| \right) \le Cr P( E,  I_r^\circ) 
 = C r \HA^{n-1}(\partial^* E \cap I_t^\circ)
$$
where $C$ is a constant depending only on $n$.
\end{cor}

\begin{proof}
The inequality follows immediately from \eqref{isso}
and the trivial estimate
$$
 \min\left(| E \cap I_r |, | I_r \setminus E| \right) \le |I_r| = 2^n r^n.
$$
(See also \cite[Remark 3.45]{A} for a different proof.)
\end{proof}





\begin{thm}[Federer Coarea Formula, Theorem 2.93  in~\cite{A}]\label{coarea2}
Let $f: \mathbb{R}^{\size} \to \mathbb{R}$ be a Lipschitz function
and $E\subset\R^n$ an  $\HA^{k}$-rectifiable set.
Then the function $t \to \HA^{k-1}(E \cap f^{-1}(t))$
is Lebesgue measurable, $E \cap f^{-1}(t)$ is 
$\HA^{k-1}$-rectifiable for almost every $t\in\R$, and 
\[
 \int_E  |\nabla_\tau f(x)| \, d \HA^{k}(x)  = \int_{t=0}^\infty \HA^{k-1}(E \cap f^{-1}(t))\,  dt
\]
where $\nabla_\tau f(x)$ is the component of $\nabla f(x)$ tangential to $E$.
\end{thm}

In the next theorem we use the following notation.
For $t\in\R$, we denote by $\Sigma_t$ the hyperplane
$\Sigma_t:=\{x\in\R^n:x_1=t\}$.
For a set $E\subset\R^n$, we denote by $E_t$
the intersection (``slice'') $E_t:=E \cap \Sigma_t$.

\begin{cor}[Coarea inequality]\label{s-fla}
Let $E\subset\R^n$ be a set with finite perimeter.
Then $\partial^* E \cap \partial I_r$ is $\HA^{\size-2}$-rectifiable for almost every $r$, and
 \begin{equation}\label{slice}
\HA^{n-1}(\partial^* E)  \geq  \int_{0}^\infty \HA^{\size-2}( \partial^* E \cap \partial I_r)  \,dr .
 \end{equation}
\end{cor}

\begin{proof} By the De Giorgi Theorem~\ref{dg} 
the reduced boundary $\partial^* E$ is $\HA^{\size-1}$-rectifiable.
We obtain the inequality in~\eqref{slice} by applying Theorem~\ref{coarea2}  
to $\partial^*E$ in place of $E$ with $k=\size-1$, $f(x)= \|x\|_{l_\infty(\R^n)}$,
and using the fact that $|\nabla_\tau f(x)| \leq 1$.
\end{proof}

\begin{thm}[Boundary slicing theorem, Theorem 18.11 in~\cite{M}]\label{Maggi}
If $E$ is a set of finite perimeter in $\mathbb{R}^\size$, 
then for almost every $t \in \R$ the slice
$
E_t = E \cap \Sigma_t
$
is a set of finite perimeter in the hyperplane $\Sigma_t\cong \R^{n-1}$ and
\[
\HA^{\size-2}( \partial^* (E_t) \Delta ( \partial^* E)_t) =0,
\] 
where $\Delta$ denotes symmetric difference of two sets
and $\partial^* (E_t)$ is the $(n-2)$-dimensional
reduced boundary of $E_t$ in $\Sigma_t$.
\end{thm}



\end{document}